\documentclass{amsart}

\usepackage{amssymb}

\newtheorem{theorem}{Theorem}[section]
\newtheorem{lemma}{Lemma}
\newtheorem{corollary}{Corollary}
\newtheorem{proposition}{Proposition}
\newtheorem{assumption}{Assumption}

\theoremstyle{definition}

\newtheorem{example}[theorem]{Example}

\theoremstyle{remark}
\newtheorem{remark}[theorem]{Remark}

\numberwithin{equation}{section}

\newcommand{\R}{\mathbb{R}}
\newcommand{\dif}{\mathrm{d}}
\newcommand{\N}{\mathcal{N}}
\newcommand{\Nat}{\mathbb{N}}
\newcommand{\Z}{\mathbb{Z}}

\newcommand{\Mat}{\tilde{\phi}}
\newcommand{\opJ}{\mathcal{J}}
\newcommand{\opL}{\mathcal{L}}
\newcommand{\dist}{\mathrm{dist}}
\newcommand{\I}{\mathcal{I}}
\newcommand{\Sph}{\mathbb{S}}

\newcommand{\Hom}{\mathsf{h}}

\newcommand{\Sch}{\mathcal{S}}
\newcommand{\F}{\mathcal{F}}

\newcommand{\ta}{\tilde{a}}

\newcommand{\EE}{\mathcal{E}_{\Omega,{\Xi}}}

\newcommand{\Manoa}{M\=anoa}
\newcommand{\Hawaii}{Hawai\kern.05em`\kern.05em\relax i }

\newcommand{\m}{{m}_0}

\begin{document}

\title[Extending error bounds for radial basis function interpolation]{Extending error bounds for radial basis function interpolation to measuring the error in higher order Sobolev norms}

\author{T.~Hangelbroek}
\address{Department of Mathematics, University of \Hawaii   -- \Manoa,
2565 McCarthy Mall, Honolulu, HI 96822, USA }
\curraddr{}
\email{hangelbr@math.hawaii.edu}
\thanks{Research of this author was supported by  by grant DMS-2010051 from the National
    Science Foundation.}

\author{C.~Rieger}
\address{Philipps-Universit\"at Marburg, 
    Department of Mathematics and Computer Science,
Hans-Meerwein-Stra\ss{}e 6, 35032 Marburg, Germany}
\curraddr{}
\email{riegerc@mathematik.uni-marburg.de}
\thanks{}


\date{}

\dedicatory{}

\begin{abstract}
Radial basis functions (RBFs) are prominent examples for reproducing kernels 
with associated reproducing kernel Hilbert spaces (RKHSs).
The convergence theory for the kernel-based interpolation in that space is well understood 
and optimal rates for the whole RKHS are often known.
Schaback added the doubling trick \cite{schaback_doubling},
which shows
that functions having double the smoothness required by the RKHS 
(along with specific, albeit complicated boundary behavior)
can be approximated 
with higher convergence rates than the optimal rates for the whole space.
Other advances allowed interpolation of target  functions which are less smooth, 
and different norms which measure interpolation error.
The current state of the art of error analysis for RBF interpolation
 treats target functions having smoothness up to twice that of the native space,
 but error measured in norms which are weaker than that required for membership in the RKHS.

Motivated by the fact that the kernels and the approximants they generate
 are smoother than required by the native space, 
this article extends the doubling trick 
to error which measures  higher smoothness.
This extension holds for a family of kernels satisfying easily checked hypotheses which 
we describe in this article, and includes many prominent RBFs.
In the course of the proof, new convergence rates are obtained for the abstract operator 
considered by Devore and Ron in \cite{devron}, and new Bernstein estimates are obtained relating
high order smoothness norms to the native space norm.
\end{abstract}

\maketitle

\section{Introduction}\label{S:Intro}
A hallmark of the mathematical theory of radial basis functions (RBFs), 
is the well-posedness of interpolation at scattered sites.
In the simplest setting, a finite set $X\subset \R^d$ generates a
 basic finite dimensional space $V_X = \mathrm{span}\{\phi(\cdot -x)\mid x\in X\}$, using 
 the RBF $\phi:\R^d\to \R$, which is a
  continuous, positive definite, radially symmetric function.
 Interpolation at the sites $X$ is well-posed; to any  function $f:\R^d\to \R$, defined at sites $X$, there is
a unique continuous RBF interpolant $I_X f\in V_X$. 
The current state of the art treats error measured in a variety of Sobolev norms up to a critical 
order determined by the RBF.
The goal of this paper is to provide a new error analysis for RBF interpolation
treating errors measured in Sobolev norms higher than this critical order.

The motivation to extend the range of error estimates to  higher  order Sobolev norms
stems from different mathematical areas.
The first motivation stems from approximation theory. 
To determine the exact range of parameters in which rigorous error estimates can be shown is a natural question, 
one which has been considered for radial basis functions in \cite{D1,GM, NWW, NWW06}.

Measuring error in higher Sobolev norms has also gained attention in the context of deep learning. 
See, for example, \cite{deryck:etal:2021,guhring2020error}.
There are several reasons to include derivative information in the loss 
function for training deep neural networks.
One motivation stems from the observation that including derivative
information can improve the performance of the predictive error in 
learning, see \cite{czarnecki2017sobolev}.
Another  stems from the fact that 
machine learning techniques 
have become an incredibly popular tool to solve partial differential equations
-- this includes using deep neural networks, but also Gaussian processes and kernels, see \cite{chen2021solving}.
This aspect is closely connected to the next motivation.

A particularly strong motivation comes from using RBFs as tools for mesh-free solution of PDEs.
In this regard, we mention 
the cosmos of {(pseudo-)spectral methods,}
although other approaches, namely
Galerkin and RBF-FD methods can also benefit. 
Traditionally, in spectral methods one considers (orthogonal) polynomials 
for which such approximation results are also available, 
see for instance  \cite[Theorem 2.2]{canuto:quarteroni:1982} where also error estimates 
in higher (weighted) Sobolev norms are discussed.
We consider now radial basis functions and pseudo-spectral methods,
see \cite{fasshauer:2005} for an introduction into the topic. 
See also \cite{Zhao:etal:2018} for non-standard differential operators.
The overall problem is that one seeks a finite dimensional approximation to a linear differential operator $\mathcal{L}$. 
A main focus in pseudo-spectral methods is that the approximate operator should be applicable 
to a function $f$ from which only its values on a discrete set of possibly scattered points $X$ are known.
A common approach is to consider a kernel-based interpolation $I_X {f}$ to the function 
and to consider the differential operator applied to the interpolant as discretized differential operator. 
To formally justify this procedure a consistency argument of the following form 
\begin{align*}
	f \approx I_{X}(f) \Rightarrow \mathcal{L}f \approx \mathcal{L}I_{X}f
\end{align*}  
is needed. 
Such estimates can be rigorously proven if the interpolation error measured in high 
(depending on $\mathcal{L}$) Sobolev norms can be controlled.

%
%
%
\subsection{The doubling trick}
For each positive definite RBF $\phi$, there is an associated 
reproducing kernel Hilbert space $\N(\phi)\subset C(\R^d)$, 
the {\em native space}, for which $(x,y)\mapsto \phi(x-y)$ is the reproducing kernel. 
The native space has an associated error analysis
for interpolation which works as follows:
the interpolation operator $I_X$ is
the $\N(\phi)$-orthogonal projector onto $V_X$. Thus $\|I_X f-f\|_{\N(\phi)} \le \|f\|_{\N(\phi)}$, which leads, thanks to the embedding
$\N(\phi)\subset C(\R^d)$ to (tautological) pointwise bounds of the form
$$|f(x)-I_X f(x)| \le P_X(x) \|f\|_{\N(\phi)},$$ where $P_X(x) = \max_{\|f\|_{\N(\phi)}=1} |I_X f(x)-f(x)|$ is called
the {\em power function}. Often, the power function can
be made small when $X$ is  well distributed near to $x$.

This natural error estimate can be improved by the {\em doubling trick} for RBFs, originally described in 
\cite{schaback_doubling}.
It is the RBF version of the classical 
Aubin-Nitsche trick \cite{aubin,nitsche}   used in the theory of Finite Elements.
Roughly, it guarantees that a function $f\in \N(\phi*\phi)\cap \N(\phi)$ 
which has deconvolution $v =(\widehat{f}/\widehat{\phi})^{\vee}$ supported in
a compact set $\Omega$
satisfies 
$$\|f-If\|_{\N(\phi)} \le \|f\|_{ \N_{\phi*\phi} } \|P_X\|_{L_2(\Omega)} .$$
A particularly strong version of this result considers $\phi$ with $\N(\phi)$ norm equivalent to $W_2^m(\R^d)$,
  $m>d/2$, and $\Omega\subset \R^d$ a compact set satisfying an interior cone condition. 
  In that case, $\N(\phi*\phi)=W_2^{2m}(\R^d)$,
and the result of applying the doubling trick gives
\begin{equation}
\label{classical}
\|f-I_Xf\|_{W_2^{\sigma}(\Omega)} \le Ch^{2m-\sigma}\|f\|_{W_2^{2m}(\R^d)}
\quad \text{for all $0\le \sigma \le m$.}
\end{equation}
Here $h:= \max_{x\in\Omega}\dist(x,X)$ is the fill distance of $X$ in $\Omega$.

Interestingly, in the case that the native space is  $W_2^m(\R^d)$, 
the RBF $\phi$, along with the finite dimensional space $V_X$,  
lies in $W_p^{\sigma}(\R^d)$,
for $\sigma<2m-d+d/p$.  
Thus, it is reasonable to ask if $I_X f$ converges to $f$ in higher order Sobolev norms when $f$ satisfies the conditions
required for the doubling trick.
This is the problem we seek to answer.

Precisely, we will show in Theorem \ref{main_interpolation} that for the interpolation error for $f=\phi*\nu\in W_2^{2m}(\R^d)$ with $\nu\in L_2(\R^d)$ having
support in $\Omega$ and with a suitable positive definite kernel
\begin{equation}\label{intro_doubling}
\|f-I_Xf\|_{W_2^{\sigma}(\Omega)}
\le 
C h^{m}q^{m-\sigma}\|\nu\|_{L_2(\R^d)}
\end{equation}
holds for  all sufficiently dense subsets and for any $m<\sigma$ for which $\lceil \sigma\rceil <2m-d/2$. 
Here we employ the separation radius
$q:= \frac12 \min_{\xi\in X}\dist(\xi,X\setminus \{\xi\})$.
Under conditions of quasi-uniformity of $X$ (i.e., when $h\le \rho q$ for some constant $\rho>1$)
and using the original result \cite{schaback_doubling}, this yields 
$$
\|f-I_Xf\|_{W_2^{\sigma}(\Omega)}
\le
C h^{2m-\sigma}\|\nu\|_{L_2(\R^d)}
$$
for all $\sigma>0$ such that $\lceil \sigma\rceil<2m-d/2$.

At this point, we note that we might have obtained such approximation orders for the case $\sigma>m$ 
by using a smoother kernel of order greater than $\sigma$ and employing the classical result (\ref{classical}).
However, following the general guideline \cite[Guideline 3.11]{schaback_wendland_2006} 
it is favorable to use the least smooth kernel to obtain a given approximation rate.

Suppose we are given a function $f\in W_2^{2m}(\R^d)$ 
supported in $\Omega$
and we would like to measure the error in the 
$W_2^{m+n}(\Omega)$ norm for $0 \le n < m-d/2$.
Standard error estimates would use the kernel $\phi_{2m}$  with $\N_{\phi_{2m}}=W_2^{2m}(\R^d)$ to construct the interpolant. 
In that case,  the resulting linear system would have condition number of order $q^{d-4m}$.
Exploiting the doubling trick of Schaback, one could use the kernel of  $W_2^{m+n}(\R^d)$ to construct an interpolation and adopt his error analysis to this case. This would still lead to a condition number of order $q^{d-2m -2n}$. 
We point out that our novel error analysis rigorously justifies to use the kernel of $W_2^{m}(\R^d)$ and hence we have to expect a condition number of order $q^{d-2m}$ which is the least one of those choices.

%
%
%
\subsection{Outline}
We introduce notation and  present  some necessary background on RBF interpolation in section \ref{SS:background}.

In section \ref{S:Bernstein} we present high order Bernstein inequalities and discuss their application
to three prominent families of RBFs: surface splines  (introduced in Example \ref{TPS_1}),
Mat{\'e}rn kernels (described  in Example \ref{Matern_example_1}),
and various compactly supported kernels including Wendland's kernels of minimal degree
(described  in Example \ref{Wendland_example_1}).

In section \ref{S:Approx} we introduce an integral-based approximation scheme
$T_X: f\mapsto T_X f\in V_X(\phi)$ and discuss its error analysis. 
The application of this approximation scheme is then discussed for surface splines, Mat{\'e}rn kernels and compactly
supported kernels.

Section \ref{S:main_PD} gives interpolation error in the case that the RBF $\phi$ is positive definite;
this setting yields the result (\ref{intro_doubling}) mentioned earlier.
This  applies to Mat{\'e}rn kernels and some compactly supported RBFs.
A precise discussion of the compactly supported kernels for which this works is given in section \ref{S:compact_support}.

Section \ref{S:main_CPD} gives interpolation error in the case that the RBF $\phi$ is conditionally positive definite. 
This  is requires a bit more care than the positive definite case; in particular, the results require extra hypotheses
which are not present in  Section \ref{S:main_PD}. 
Section \ref{SS:algebraic} shows how these hypotheses can be met for an RBF whose Fourier transform
has an algebraic singularity.
 Section \ref{SS:surface_splines} treats the surface splines, and derives estimates in terms of the fill distance.
 
 In Appendix \ref{appendix}, we prove a lemma about regularity of local polynomial reproductions which is
 used in section \ref{S:Approx}, but which may find use beyond the scope of this article.
 
%
%
%
%
\section{Notation and Background}
\label{SS:background}
Throughout the article, we will use $C$ as a generic positive constant whose value
may change from line to line.

 Denote by $x\mapsto |x|$  the Euclidean norm in $\R^d$  and  let $(x,y)\mapsto \dist(x,y)=|x-y|$ be the corresponding distance. 
Let $\Omega\subset \R^d$ and 
$\Xi\subset \Omega$ a finite subset.
Define the {\em separation radius} 
by
$q_{\Xi}:= \frac12 \min_{\xi\in {\Xi}}\dist(\xi,{\Xi}\setminus \{\xi\})$,
the {\em fill distance} %
by 
$h_{\Xi,\Omega}:=\sup_{x\in\Omega} \dist(x,\Xi)$, 
and
the {\em mesh ratio} 
by
$\rho_{\Xi,\Omega}:=\frac{h_{\Xi,\Omega}}{q_\Xi}$.
When the underlying sets are clear from context, 
we will simply write $h$, $q$ and $\rho$.

Define the space of polynomials by  $\mathcal{P}(\R^d)$ and the subspace of polynomials of degree $m$ or less by $\mathcal{P}_m(\R^d)$.
The space of Schwartz functions is denoted $\mathcal{S}(\R^d)$.
 The Fourier transform of a Schwartz function is 
 $$\F{\gamma}(\omega) = (2\pi)^{-d/2} \int_{\R^d} \gamma(x) e^{-ix\cdot \omega}\dif x.$$
and for tempered distributions, the Fourier transform $\F u$ is the distribution which satisfies
$\langle\F u,\gamma\rangle :=\langle u,\F \gamma\rangle$ for all $\gamma\in\Sch(\R^d)$.
For $m\ge0$, 
define  
$$
\mathcal{S}_{m}(\R^d) 
:=
\{\gamma\in \mathcal{S}(\R^d)\mid \sup_{\omega\in \R^d} |\omega |^{-\m}\gamma(\omega)<\infty\}.
$$
If the distributional Fourier transform of $\phi$ coincides
on $\R^d\setminus\{0\}$ with a measurable function which represents the
 Fourier transform
on  
$\mathcal{S}_{2m}(\R^d)$
then it has {\em generalized Fourier transform} of order $m$.
Denoting the generalized Fourier transform of $\phi$  
by $\widehat{\phi}:\R^d\setminus\{0\}\to \mathbb{C}$, 
the above definition is equivalent to 
 the identity
$$
(\forall \gamma\in \Sch_{2m}(\R^d))\quad
\int_{\R^d} 
\widehat{\phi}(\omega) \gamma(\omega) 
\dif \omega
= 
\int_{\R^d} 
\F\gamma(\omega)
 \phi(\omega)
\dif \omega.
$$
See \cite{GV,Jones} for more background.
\smallskip

\paragraph{\bf Sobolev spaces}
We recall (see for instance \cite[Definition 2.39]{Abels}) the Bessel potential operators which are defined for tempered
 distributions via the formula
 $$
\F(\opJ_s f)
 =  
  (1+|\omega|^2)^{s/2}
  \F f.$$
 For $\tau \ge0$ we define the Sobolev space $H^\tau $ via Bessel potentials: 
 that is, $H^\tau$ is the space of  all $u\in L_2(\R^d)$  such that 
$ \opJ_s u\in L_2(\R^d)$. 
 Its norm is
 $$\|u\|_{H^{\tau}}^2 := 
 \int_{\R^d}
 \F u(\omega)
 (1+|\omega|^2)^{\tau/2}\dif \omega.$$
 For any $s,\tau\in\R$,
$\opJ_s : H^{\tau} \to H^{\tau-s}$
 is an isometry between Sobolev spaces.
 
The space $\dot{H}^\tau$ consists of distributions $u$ for where there exists $p\in \mathcal{P}$
such that $(-\Delta)^{\tau/2} u- p \in L_2(\R^d)$. 
If $u$ has generalized Fourier transform of some order, then 
$$
|u|_{\dot{H}_2^{\tau}}
:= 
\bigl\||\cdot|^{\tau} \widehat{u}\bigr\|_{L_2(\R^d)}.
$$

 In order to  work on compact sets $\Omega\subset \R^d$
we also need Sobolev spaces on domains.
For  $k\in\Nat$, we define the Sobolev space $W_2^k(\Omega)$  
to be all  functions $u$ with distributional derivatives $D^{\alpha} u \in L_2(\Omega)$
 for all $|\alpha| \le k$. 
Associated with these spaces are the 
seminorms
$$|u|_{W_2^k(\Omega)}^2: = \sum_{|\alpha|=k} \frac{k!}{\alpha!}\|D^{\alpha}u\|_{L_2(\Omega)}^2$$
and norms $\|u\|_{W_2^k(\Omega)}^2:=\sum_{j=0}^k |u|_{W_2^k(\Omega)}^2$.
For fractional order Sobolev spaces, we use the seminorm
$$
|u|_{W_2^{k+s}(\Omega)}^2 
= 
\sum_{|\alpha|=k}
 \int_{\Omega}\int_{\Omega}
 	\frac{|D^{\alpha}u(x)-D^{\alpha}u(y)|}{|x-y|^{d+2s}}
\dif x\dif y
$$
and norm $\|u\|_{W_2^{k+s}(\Omega)}^2:=\sum_{j=0}^k |u|_{W_2^k(\Omega)}^2+|u|_{W_2^{k+s}(\Omega)}^2 $.
It is well known that $W_2^{\tau}(\R^d)$ and $H^{\tau}$ have equivalent norms.  To ensure equivalence between 
seminorms, we use the following lemma.
\begin{lemma} 
\label{Sob_equiv}
If $u$ has generalized Fourier transform of order $\lfloor\tau\rfloor/2$, then the seminorms
$|u|_{W_2^{\tau}(\R^d)}$ and $ |u|_{\dot{H}^{\tau}}$ are equivalent.
\end{lemma}
\begin{proof}
Note that  $D^{\alpha}\gamma\in \Sch_{|\alpha|}$ for every $\gamma\in \Sch$,
so  if $u$ has generalized Fourier transform of order $\sigma$ and $|\alpha|\ge \sigma/2$ 
then
$$\int_{\R^d} (i\omega)^{\alpha} \widehat{u}(\omega) \widehat{\gamma}(\omega) \dif \omega= \int_{\R^d}  \bigl(D^{\alpha} u(x) \bigr)\gamma(x)\dif x$$ holds for all 
$\gamma\in \Sch$.
So the distributional Fourier transform of $D^{\alpha} u$ 
is  represented by the locally integrable function
$\F(D^{\alpha} u )(\omega)=  (i\omega)^{\alpha} \widehat{u}(\omega)$.
%

In case  $\tau=k\in\Nat$,
it follows that for $u$ with generalized Fourier transform of order $k/2$, we have
$\int_{\R^d} |D^{\alpha} u(x)|^2 \dif x = \int_{\R^d} |\omega^{\alpha}\widehat{u}(\omega)|^2\dif \omega$, so 
$|u|_{\dot{W}_2^k(\R^d)} \sim  |u|_{\dot{H}^{k}}$.

In case $\tau\notin \Nat$ with $s=\tau-\lfloor \tau\rfloor$,
we use  \cite[Proposition 3.4]{Hitch}
to show that, for a multiindex $|\alpha|=\lfloor \tau\rfloor$,
$|D^{\alpha} u|_{\dot{H}^s} \sim |D^{\alpha} u|_{W_2^s(\R^d)}$.
Applying this to the definition gives 
$
|u|_{W_2^{\tau}(\R^d)}^2 
\sim \sum_{|\alpha|=\lfloor \tau\rfloor} |D^{\alpha} u|_{\dot{H}^s}^2$.
The norm equivalence
$ \sum_{|\alpha|=\lfloor \tau\rfloor} |D^{\alpha} u|_{\dot{H}^s}^2
\sim
|u|_{\dot{H}^{\tau}}^2$
follows because $\F(D^{\alpha} u )(\omega)=  (i\omega)^{\alpha} \widehat{u}(\omega)$ for $\omega\neq 0$.
\end{proof}

 %
 %
 %
 %
 \paragraph{\bf Radial basis functions and native spaces}
 A function $\phi:\R^d\to \R$ is {\em conditionally positive definite of order $\m$} (hereafter abbreviated by CPD)  
 if the following holds: for any finite
$\Xi\subset \R^d$,
the {\em collocation matrix}
$$\Phi_{\Xi}  = \bigl(\phi(\xi-\zeta)\bigr)_{\xi,\zeta\in \Xi}$$
is strictly positive definite on the subspace 
$$ \{a\in \R^\Xi \mid (\forall p\in \mathcal{P}_{\m-1})\, \sum_{\xi\in \Xi} a_{\xi} p(\xi)=0\}.$$
If $\m\le 0$, then $\phi$ is {\em  positive definite}.
A  function which is CPD and symmetric with respect to rotations  is called
a {\em radial basis function} (RBF).

 A  version of Bochner's theorem (\cite[Theorem 8.12]{Wend})  asserts that if $\phi$ is continuous and increases at most algebraically 
 (so that $|\phi(x)|\le C |x|^{\alpha}$ for some $\alpha \in \R$) 
 and if $\phi$ has a continuous generalized Fourier transform of order $\m$ which
 satisfies that $\widehat{\phi}>0$ on some open set, then $\phi$ is CPD of order $\m$. Although
more will come later, this assumption will be in place throughout the article. 
It is worth noting that both the order of the generalized Fourier transform 
and the order of conditional positive definiteness has the nesting property:
 if $\phi$ has order $\m$ then it has order $\m+1$.
 
 For a CPD function of order $\m$ there is an associated function space, 
called the {\em native space},
$\N(\phi)$
which consists of continuous functions.
One may find its construction in \cite{SchNative,Wend}.
The space has a semi-inner product 
$(f,g)\mapsto \langle f,g\rangle_{\N(\phi)}$
with nullspace $\mathcal{P}_{\m-1}$.
It is complete in the sense that the quotient
$\N(\phi)/\mathcal{P}_{\m-1}$ is a Hilbert space. We denote the induced seminorm
by $f\mapsto |f|_{\N(\phi)}$.

If  $\phi$ is positive definite (i.e., $\m\le 0$), the nullspace is trivial, and $\N(\phi)$ is a
Hilbert space.
In this case,  $f\mapsto | f|_{\N(\phi)}$ is a norm.

It is worth noting that the native space depends both on the function $\phi$ as well as the order $\m$;
 this is relevant because of the nesting property described above, so a given CPD function will generate infinitely many
native spaces 
(one for each order, although they only differ by the polynomial space $\mathcal{P}_{\m-1}$).

For any functional of the form 
$\sum_{\xi\in \Xi} a_\xi \delta_{\xi}$ 
supported on 
 $\Xi\subset \R^d$
which annihilates $\mathcal{P}_{\m-1}$ we have for  $f\in \N(\phi)$ that
%
\begin{equation}
\label{RK}
\sum_{\xi\in\Xi} a_\xi f(\xi) 
= 
\left\langle 
   f, \sum_{\xi\in\Xi} a_\xi \phi(\cdot -\xi)
\right\rangle_{\N(\phi)}.
\end{equation}
For  
$\Xi\subset \R^d$, 
we define the finite dimensional  space 
$$
V_\Xi(\phi) 
:= 
\bigl\{
  \sum_{\xi\in\Xi} a_\xi \phi(\cdot-\xi)
\mid 
  \sum_{\xi\in\Xi} a_\xi \delta_{\xi}\perp \mathcal{P}_{\m-1}
\bigr\}
+\mathcal{P}_{\m-1}.
$$
For any 
$\Xi \subset \R^d$, we have  $V_\Xi(\phi)\subset \N(\phi)$.

If 
$\Xi$ 
is {\em unisolvent} with respect to $\mathcal{P}_{\m-1}$ 
(meaning that if $p\in \mathcal{P}_{\m-1}$ vanishes
on 
$\Xi$, 
then it is identically zero), 
then the 
interpolation operator
$$I_\Xi:\N(\phi)\to V_\Xi (\phi) \quad \text{where} \quad (I_\Xi f)|_\Xi = f|_\Xi \text{ for all }f\in \N(\phi)$$
is well-defined. 
It is the orthogonal projector onto 
$V_\Xi(\phi)$
with respect to the $\N(\phi)$ semi-inner product.
Note that, like the native space, the interpolation operator depends
on the order $\m$ of conditional positive definiteness 
 as well as  on $\phi$ (as well as on 
 $\Xi$).

For a CPD function $\phi$ which has a continuous generalized Fourier transform of order $\m$, 
the native space can be expressed as the space of continuous functions $f$ which 
are tempered distributions, which have a generalized Fourier transform of order $\m/2$,
and for which $\int_{\R^d} |\widehat{f}(\omega)|^2/\widehat{\phi}(\omega) \dif \omega<\infty$.
 In this case, the formula 
\begin{equation}
\label{Fourier_NS}
\langle f,g\rangle_{\N(\phi)} 
= 
\int_{\R^d} 
   \widehat{ f}(\omega)\overline{\widehat{g}(\omega)} (\widehat{\phi}(\omega))^{-1} 
\dif \omega
 \end{equation} 
holds. See \cite{SchNative,Wend} for details.

%
%
%
%
%
\section{Higher order Bernstein inequalities}
\label{S:Bernstein}
Bernstein estimates for RBF approximation
have been demonstrated in \cite{NWW06},
and more recently  \cite{HNRW} for bounded regions.
The existing literature treats the case that the weaker norm is $L_2(\R^d)$.
In this section we present Bernstein inequalities where the weaker norm is the native space. 
These hold for RBFs which have the following property:
\begin{assumption}
\label{A:CPD}
We assume $\phi$ to be  an RBF
whose  generalized Fourier transform satisfies
\begin{eqnarray*}
C_1(1+|\omega|^2)^{-\tau}
&\le& 
\widehat{\phi}(\omega) 
\qquad 
\text{for all }\omega\in\R^d\setminus\{0\}\\
\widehat{\phi}(\omega)
&\le& 
C_2|\omega|^{-2\tau}
\qquad 
 \text{for all }|\omega|>r_0.
\end{eqnarray*}
 for some exponent
 $\tau>d/2$ and constants  $0<C_1\le C_2$,
\end{assumption}
This guarantees the continuous embedding $H^{\tau}\subset\N(\phi)$.
However,  it does not quite
imply $\N(\phi) \subset \dot{H}^{\tau}$, 
since $\widehat{\phi}$ may have a sharper singularity at $\omega=0$ than $\mathcal{O}(|\omega|^{-2\tau})$.

Under this assumption, \cite[Theorem 12.3]{Wend} applies (see also \cite{NW91}) , with  a $\tau$ dependent constant $C_0$:
\begin{equation}\label{cpd_eigenval}
\lambda_{\min} (\Phi_{\Xi})
:= 
\min_{\sum_{\xi\in\Xi}^N a_\xi \delta_{\xi} \perp \mathcal{P}_{\m-1}}  \sum_{\xi,\zeta\in\Xi}
        a_\xi a_\zeta \phi(\xi- \zeta) \ge C_0  C_1q^{2\tau-d} \|a\|_{\ell_2(\Xi)}^2,
\end{equation}
This can be used to prove a bandlimited approximation result as in
\cite[Lemma 3.3]{NWW06}.
To this end, for $\sigma>0$, define for a tempered distribution $u$, the function $u_{\sigma} := (\widehat{u} \chi_{B(0,\sigma)})^{\vee}$.
By the identity (\ref{Fourier_NS}), if $u\in \N(\phi)$, then $u_{\sigma}\in \N(\phi)$ as well.
\begin{lemma}\label{bandlimiting}
If $\phi$ satisfies Assumption \ref{A:CPD}, then there is $\kappa>0$ so that
for any finite set of points $\Xi \subset \R^d$, if $\sigma >\max(r_0,\kappa/q)$, then
$|u-u_{\sigma}|_{\N(\phi)} \le \frac12  |u|_{\N(\phi)}$ for all $u\in V_{\Xi}(\phi)$.
\end{lemma} 
\begin{proof}
The proof follows that of \cite[Lemma 3.3]{NWW06}, with a simple modification to 
treat the requirement that $\sigma>r_0$.
\end{proof}

As in \cite[Theorem 5.1]{NWW06} this gives rise to a Bernstein estimate. 
In contrast to the result in \cite{NWW06},
 this uses  a higher order smoothness norm on the right hand side.

\begin{theorem}\label{cpd_euclidean}
 Suppose $\phi$ satisfies Assumption \ref{A:CPD}, and  that $0\le s<\tau-d/2$.
 Then there is a constant $C$ so that
 for any $\Xi \subset \R^d$ with separation radius $q<1$,
 $$
 |\opJ_s u |_{\N(\phi)} \le C q^{-s} |u|_{\N(\phi)}
 $$
holds for any 
$u\in V_\Xi(\phi)$.
 \end{theorem}
Note that 
every  polynomial space  $\mathcal{P}_{m}$ is invariant under $\opJ_s$.
Since $ \mathcal{P}_{\m-1}$ is the nullspace of the native
space semi-norm, we have $|\opJ_s(u+p)|_{\N(\phi)} = |\opJ_s u|_{\N(\phi)}$. 
Thus for $\phi$ satisfying
Assumption 1, 
$|\opJ_s u |_{\N(\phi) }$ can be calculated, via (\ref{Fourier_NS}), as an integral on the Fourier domain.
 \begin{proof}
 The  function 
 $\psi_{\tau-s}:=\opJ_{2s} \phi $
  satisfies Assumption \ref{A:CPD}, with $\tau-s$ in place of $\tau$,
 as can be observed from its  
 generalized Fourier transform 
 $\widehat{\psi_{\tau-s} }(\omega) = (1+|\omega|^2)^{s}\widehat{\phi}(\omega)$.
 An application of (\ref{Fourier_NS}) gives
 $|\opJ_s u|_{\N(\phi)}^2=\int_{\R^d}
  |\sum_{\xi\in\Xi} c_\xi e^{i\langle \omega,\xi\rangle}|^2 (1+|\omega|^2)^{s} \widehat{\phi}(\omega)
d\omega$
 for any $u\in V_\Xi(\phi)$,
 which provides the identity
\begin{eqnarray*}
\left|\opJ_s u
\right|_{\N(\phi) }^2
& =&
\left|
\sum_{\xi\in\Xi} c_{\xi} \psi_{\tau-s}(\cdot -\xi)
\right|_{\N(\psi_{\tau-s})}^2 
= 
|\tilde{u}|_{\N(\psi_{\tau-s} )}^2,
\end{eqnarray*} 
where we define 
$\tilde{u} := \sum_{\xi\in \Xi} c_\xi \psi_{\tau-s}(\cdot -\xi)$.
Let $\sigma = 2\max(\kappa /q, r_0)$. 
Then Lemma \ref{bandlimiting} 
guarantees that 
$
|\tilde{u}|_{ \N(\psi_{\tau-s} ) }^2
 \le 
 2 |(\tilde{u})_{\sigma} |_{ \N(\psi_{\tau-s} )}^2.
$
Finally, we have
\begin{eqnarray*}
 |
 \left(\tilde{u} \right)_{\sigma}
 |_{\N(\psi_{\tau-s} ) }^2 
 &=& 
 \int_{|\omega|<\sigma} 
 \Bigl|\sum_{\xi\in\Xi} c_\xi e^{i\langle \omega,\xi \rangle}\Bigr|^2\widehat{\phi}(\omega)(1+|\omega|^2)^{s}\, \dif \omega \\
 &\le& 
 \frac{1+4\kappa^2}{4\kappa^2}
 \sigma^{2s} 
 \int_{|\omega|<\sigma} 
 \Bigl|\sum_{\xi\in\Xi} c_\xi e^{i\langle \omega,\xi\rangle}\Bigr|^2 \widehat{\phi}(\omega)
 \dif \omega
\le 
C q^{-2s}
 \left| u \right|_{\N(\phi)}^2, 
\end{eqnarray*}
%
%
%
since  $q<1$.
In the first inequality, we have used the fact that $\sigma$ is bounded
below by $\sigma\ge 2 \kappa>0$,
so $1+\sigma^2 \le   \left(\frac{1+4\kappa^2}{4\kappa^2}\right)\sigma^2$. %
The second inequality follows automatically if $\sigma= 2\kappa/q$;
if  $r_0>\kappa/q$, then the result holds with a slightly larger, $r_0$ dependent constant, 
because 
$\sigma^{2s} \le (2r_0)^{2s} \le
 (2r_0)^{2s} q^{-2s}
 $.
\end{proof}

This applies to a number of prominent RBF families.
\begin{example}
\label{TPS_1}
The surface spline $\phi_m$, the fundamental solution to $\Delta^m$ on $\R^d$
 for $m>d/2$,
is CPD of order $\lfloor m-d/2\rfloor +1$, and has generalized
Fourier transform $\widehat{\phi_m}(\omega)=|\omega|^{-2m}$
of order $\lfloor m-d/2\rfloor +1$. 
It is, however, customary to consider $\phi_m$ as CPD of order $\m=m$,
in which case \cite[Theorem 10.43]{Wend}  shows that the native spaces is the Beppo-Levi space
$$\mathrm{BL}_m(\R^d)=
\{f\in L_{2,loc}(\R^d)\mid (\forall |\alpha|=m)\ D^{\alpha}f\in L_2(\R^d)\}.
$$ 
The  seminorm for this space is 
$ |f |_{\N(\phi_m)} = |f|_{W_2^{m}(\R^d)}$.
%
Then Theorem \ref{cpd_euclidean} states that
for $u\in V_{\Xi}(\phi_m)$
and  $s<m-d/2$, we have
 $|\opJ_s u|_{\dot{H}^{m}} \le C q^{-s} |u|_{W_2^m(\R^d)}$. 
\end{example}
\begin{example}
\label{Matern_example_1}
The Mat{\'e}rn kernels $\Mat_{\tau}$,  $\tau>d/2$,
known also as Bessel potential kernels,
are
 the fundamental solutions to the (possibly) fractional operator
$(1-\Delta)^{\tau}$ on $\R^d$.
They
are  strictly positive definite, with native space $\N(\Mat_{\tau}) =H^{\tau}$. 
For any
$u\in V_{\Xi}(\Mat_{\tau})$
we have $\|u\|_{H^{\tau+s}} \le C q^{-s} \|u\|_{H^{\tau}}$
as long as $s<\tau-d/2$.
These 
RBFs are discussed further
in Example \ref{Matern_example_2}.
\end{example}
\begin{example}
\label{Wendland_example_1}
Various compactly supported RBFs, 
including Wendland's compactly supported RBFs of minimal degree, denoted $\phi_{k,d}$ (where 
$k$ is a parameter derived from its construction, but related to its smoothness) satisfy Assumption 1.
Each kernel $\phi_{k,d}$ is  strictly positive definite, and has  native space $\N_{\phi_{k,d}} = H^{k+(d+1)/2}(\R^d)$.
Theorem \ref{cpd_euclidean} states that for any
$u\in V_{\Xi}(\phi_{k,d})$
we have 
$$\|u\|_{H^{k+(d+1)/2+s}(\R^d)} \le C q^{-s} \|u\|_{H^{k+(d+1)/2}(\R^d)}$$
as long as 
$s< k+1/2 -d/2$.
These are discussed again in Example \ref{Wendland_example_2} in the next section.
\end{example}

%
%
%
%
%
%
\section{RBF Approximation with Sobolev norms}
\label{S:Approx}
We now give Jackson estimates for the spaces $V_{\Xi}(\phi)$ using the norm 
$\| \cdot \|_{H^{m+n}(\R^d)}$, with $0\le n<m-d/2$. 
Our first result involves a version of the approximation scheme developed
in \cite{devron} for RBFs which are fundamental solutions to differential operators.
This scheme was used to get 
approximation results with error measured in $L_p(\R^d)$; 
we expand this slightly to error in Sobolev norms, and for RBFs satisfying a more general set of conditions.
Specifically, we show that it  provides strong results for target functions 
$u\in H^{2m}$
having deconvolution $(\widehat{u}/\widehat{\phi})^{\vee}$ supported in $\Omega$.

For this, we make a basic assumption about the radially symmetric function $\phi$. Namely, 
that it is a smooth perturbation of a type of (essentially) homogeneous function.
 To make this definition we introduce the function $\Hom_s$ for $s\ge 0$ as 
$$\Hom_s(x)=\begin{cases} |x|^s& s\notin 2\Nat\\ |x|^s\log|x| & s\in 2\Nat.\end{cases}$$
By \cite[(3.1)]{HL}, it follows that
\begin{equation}\label{hom_deriv}
D^{\alpha} \Hom_s(x) 
= 
p_{s-|\alpha|}(x)\log(x) +q_{s-|\alpha|}(x)
\end{equation} 
with 
$q_{s-|\alpha|}$ a homogeneous, rational function of degree ${s-|\alpha|}$,
and
$p_{s-|\alpha|}$ a homogeneous polynomial of degree $s-|\alpha|$, which is zero when $s\notin 2\Z$ or when $s-|\alpha|<0$.

\begin{assumption}\label{Assumption2}
Suppose $s$ and $L$ are positive, with $s>d/2$ and $L>s+d$.
We assume $\phi \in  C(\R^d)\cap C^{s+d-1}(\R^d\setminus\{0\})$ is  radially symmetric, and
 there is a constant $r_0>0$  so that the following two conditions hold
 \begin{enumerate}
 \item there is a constant $C$ so that  for all multi-indices $|\beta|=L$ and $|x|>r_0$,
$$
|D^{\beta} \phi(x)| 
\le 
C |x|^{s-|\beta|}$$
\item there exist  functions $u,v\in C^L(\overline{B(0,r_0)})$ so that for $|x|<r_0$
$$
\phi(x) 
= 
u(x) + \Hom_s(x)v(x).
$$
\end{enumerate}
\end{assumption}
This assumption guarantees that 
$\phi$ is a smooth perturbation of $\Hom_s$, which has distributional
Fourier transform $\widehat{\Hom}_s(\xi) \propto |\xi|^{-s-d}$ on $\R^d\setminus\{0\}$.
Although neither Assumption 1 nor 2 implies the other, if $\phi$ is to satisfy both simultaneously,
it must follow that  $s=2\tau-d$.
\begin{example}\label{ss_example_2}
The  family of surface splines given in Example \ref{TPS_1}
are defined  for $m>d/2$ as 
$\phi_m(x) = C_{m,d} \Hom_{2m-d}(x)$. 
Thus, they  satisfy Assumption 2 with $s=2m-d$. 
Item 1 follows from (\ref{hom_deriv}) and the remark following it,
while item 2 follows with $u=0$ and constant $v$.
\end{example}
\begin{example} \label{Matern_example_2}
The Mat{\'e}rn kernels 
$\Mat_{\tau}(x) = |x|^{\tau-d/2} K_{\tau-d/2}(|x|)$ satisfy Assumption 2 with $s=2\tau-d$.
Here  $K_{\mu}$ is a  modified Bessel function (\cite[10.25]{DLMF}).
Each $\Mat_{\tau}$ is in $C^{\infty}(\R^d\setminus \{0\})$ and 
satisfies the decay condition 
$|D^{\alpha}\Mat_{\tau}(x) |\le C_M |x|^{-M}$ for all $M$ and all $\alpha$. 
Furthermore, 
item 2 holds by using the convergent power series expansion 
$$
\Mat_{\tau}(x) 
= 
\sum_{j=0}^{\infty}a_j |x|^{2j} + \Hom_{2\tau-d}(|x|) \sum_{j=0}^{\infty} b_j |x|^{2j}
$$
which is valid
for all $\mu>d/2$.
 When $\mu-d/2 \in \Nat$, this is given in both \cite[10.31.1]{DLMF} and \cite[9.6.11]{AS}.
When $\mu-d/2$ is fractional, it follows from either 
\cite[10.27.4/10.25.2]{DLMF}  or \cite[9.6.2/9.6.10]{AS}.
\end{example}
\begin{example} \label{Wendland_example_2}
The compactly supported Wendland kernels of minimal degree $\phi_{k,d}$, 
described in \cite[Chapter 9]{Wend}  satisfy  Assumption 2 only  in dimension
$d=2$.  
Indeed, for $d\in\Nat$, $\phi_{k,d}\in C^{2k+\lfloor d/2\rfloor +1}(\R^d\setminus\{0\})$,
so when $d=2$, 
$\phi_{k,2}\in C^{2k+2}(\R^2\setminus\{0\})$. In this case, $s=2k+1$ and $r_0=1$.
 Item 1 holds because $\mathrm{supp}(\phi_{k,d})=B(0,1)$. 
 
The fact that  item 2 holds follows  from \cite[Theorem 9.12]{Wend}. 
Specifically, $\phi_{k,2}(x) =p(|x|)$ for a polynomial $p(r)=\sum_{j=0}^{3k+2}a_j r^j$
 whose first $k$ odd coefficients are zero. I.e., $a_{2k+1}$ is the first nonzero  coefficient of an odd power.
 
 There do exist a number of compactly supported RBFs which satisfy Assumptions 1 and  2, however (the {\em generalized
 Wendland functions} studied in \cite{CH}).
 These are discussed in section \ref{S:compact_support}.
\end{example}
Our interest is to approximate functions $f$ having the form $f=\phi*\nu+p$, for $\nu\in L_2(\R^d)$, 
with $\mathrm{supp}(\nu)$
contained in a compact set $\Omega$, and $p\in\mathcal{P}_{\m-1}$.
We note that  $\phi$ is sufficiently smooth to allow differentiation under the integral sign:
$$
D^{\alpha} \int_{\Omega} \nu(z)\phi(x-z)\dif z
=
\int_{\Omega} \nu(z) D^{\alpha} \phi(x-z)\dif z
$$
whenever $|\alpha|< s+d$, by compactness of $\Omega$, integrability of $\nu$, and continuity of $D^{\alpha}\phi$.

\subsection{Approximation scheme}
We consider an approximation scheme similar to the on presented in  \cite{devron}. For this, we
consider a compact set $\Omega\subset \R^d$, a finite subset
$\Xi\subset \Omega$, and a sufficiently regular {\em local polynomial reproduction}. 
The latter is a map $a(\cdot,\cdot): \Xi \times \Omega \to \R$ which satisfies the following conditions:
\begin{itemize}
\item for every $z\in \Omega$ if $\dist(\xi,z)>Kh$ then $a(\xi,z)=0$
\item for every $z\in \Omega$, $\sum_{\xi\in \Xi} \bigl| a(\xi,z) \bigr|\le \Gamma$
\item for every $p\in \mathcal{P}_{L}$ and $z\in \Omega$, $\sum_{\xi\in \Xi}a(\xi,z) p(\xi) =p(z)$
\item  for every $\xi\in \Xi$, $a(\xi,\cdot)$ is measurable.
\end{itemize}
If $\Omega$ satisfies an interior cone condition and $h$ is sufficiently small,
then \cite[Theorem 3.14]{Wend} guarantees existence of a local polynomial reproduction which
has the first  three of these four properties. 
In the appendix, we present the modification to  \cite[Theorem 3.14]{Wend}  which is needed  to get the fourth condition
 (actually, we show that each $a(\xi,\cdot)$ can be chosen to be infinitely smooth).

For any function $f$ which can be decomposed as  $f=\phi*\nu +p$, with $\nu\in L_2(\R^d)$ 
having support in $\Omega$,  and $p\in \mathcal{P}$,
we define the approximation scheme $T_\Xi$ as
$$T_{\Xi}f(x) := \sum_{\xi \in \Xi} \left(\int_{\Omega} a(\xi,z) \nu(z)\dif z\right)\phi(x-\xi) + p(x).$$
\begin{remark}
\label{annihilation_remark}
 If $L\ge m$ and  $\nu\perp \mathcal{P}_{m}$, then the coefficients 
 $A_{\xi} =  
 \int_{\Omega} a(\xi,z) \nu(z)\dif z$
 satisfy 
 $$\sum_{\xi\in\Xi} A_{\xi} p(\xi) 
 = 
  \int_{\Omega}\nu(z) \sum_{\xi\in\Xi} a(\xi,z)p(\xi) \dif z
= \int_{\Omega} \nu(z) p(z) \dif z = 0$$ 
 for any $p\in \mathcal{P}_{m}$.
 In particular, if $\phi$ is CPD of order $\m$ and
 $f=\nu*\phi+p$ with
 $\nu\perp \mathcal{P}_{\m-1}$ and $p\in \mathcal{P}_{\m-1}$, then
  we have $T_{\Xi} f\in V_{\Xi}(\phi)$.
 \end{remark}

\subsection{Approximation error}

In order to calculate the error $\|D^{\alpha} f - D^{\alpha} T_{\Xi}f\|_{L_2(\R^d)}$, 
we introduce, for each multiindex
 $\alpha$ with $|\alpha| < s+d$,
 the error kernel $E^{(\alpha)}:\R^d\times \Omega \to \R$, where
$$
E^{(\alpha)}(x,z):= \left|D^{\alpha}\phi (x - z)  -   \sum_{\xi\in \Xi}  a(\xi, z) D^{\alpha}\phi (x - \xi) \right|.
$$
To analyze the error kernel, we make use of polynomial reproduction 
in the following way: 
\begin{lemma}
\label{Taylor_Lemma}
Suppose that $w\subset \R^d$, $W$ is a neighborhood of $w$,
 $\tilde{X}\subset W$ is a finite set, 
 and 
 $\ta \in \R^{\tilde{X}}$
 satisfies $\sum_{\zeta \in \tilde{X}} \ta_{\zeta} p(\zeta) =p(w)$
 for all $p\in \mathcal{P}_{L}$, along 
 with $|\zeta-w|>r \Rightarrow \ta_{\zeta}=0$.
If $U$ is $M$-times continuously differentiable in a neighborhood of
$\overline{B(w,r)}$
with $M\le L+1$, 
 then   we have
 \begin{equation}
\label{functional_bound}
\Bigl|
	U(w) - \sum_{\zeta \in \tilde{X}} \ta_{\zeta} U(\zeta)
\Bigr|
\le   
\frac{\| \ta \|_{\ell_1(\tilde{X})}}{M!} 
 r^{M} \max_{|\beta|=M} \|D^{\beta} U\|_{L_{\infty}(B(w,r))}.
\end{equation}
\end{lemma}
 \begin{proof}
We can express $U(\zeta) =P(\zeta) +R(\zeta)$, with  $P$ the Taylor polynomial 
of degree $M-1$ centered at $w$. Thus,
$
P =
\sum_{|\beta| < M} \frac{1}{\beta!}D^{\beta}U(w) (\cdot-w)^{\beta}
$.
For $\zeta\in B(w,r)$, the remainder satisfies
$$
 |R(\zeta) | 
 \le 
 \frac{1}{M!} |\zeta-w|^{M} 
 \max_{|\beta|=M} \|D^{\beta} U\|_{L_{\infty}(B(w,r))}.
 $$
Then 
 $
 |U(w) - \sum_{\zeta\in \tilde{X}} \ta_\zeta U(\zeta)|
 \le 
 \|\ta\|_{\ell_1} \max_{|\zeta-w|\le r} |R(\zeta)|
 $, 
 and the result follows.
 \end{proof}
\begin{lemma}
\label{far_field}
Suppose $\phi$ satisfies Assumption 2.
Then the error kernel
satisfies, for $|x-z|>2Kh$,
the estimate
 $$E^{(\alpha)}(x,z)\le
  \begin{cases} 
   C h^{s-|\alpha|}\left(\frac{|x-z|}{h}\right)^{s-L} 
   & |x-z|\notin [r_0-Kh ,  r_0+Kh],\\
   Ch^{s+d-1-|\alpha|} 
   &r_0-Kh\le |x-z|\le r_0+Kh.
 \end{cases}
 $$
 \end{lemma}
  \begin{proof}
 We split this into three cases according to the  size of $|x-z|$.
Case 1 treats the punctured space $|x-z|>r_0+Kh$,
 Case 2 treats the annulus
  $r_0 -Kh\le |x-z|\le r_0+Kh$,
 and  Case 3 treats the inner annulus  $2Kh\le |x-z|<r_0+Kh$.
 
 In each case, we use Lemma \ref{Taylor_Lemma} 
 applied to $U= D^{\alpha}\phi(x-\cdot)$ at the point $w= z$
 in $W=\Omega$
 using the point set $\tilde{X}=\Xi$ 
 and the vector
 $\ta \in\R^{\tilde{X}}$
defined by
 $\ta_{\zeta} = a(\zeta,z)$.
 By local polynomial reproduction, 
  the hypotheses of Lemma \ref{Taylor_Lemma}  hold with $r=Kh$ and $\|\ta\|_{\ell_1(\tilde{X})}=\Gamma$.
 
 The only difference between the cases lies in the smoothness $M$ enjoyed by $U$.
 
 \smallskip

\noindent{\em Case 1:} Assume $|x-z|>r_0+Kh$. 
In this case, 
$U= D^{\alpha}\phi(x-\cdot)$ is $M=L-|\alpha|$  times continuously differentiable on 
$ \R^d\setminus \overline{B(x,r_0)}$.
 Under these conditions, we have 
 $E^{(\alpha)}(x,z)
 =
 |U(w) -   \sum_{\zeta \in \tilde{X}} \tilde{a}_{\zeta}
 U(\zeta)|$, 
 so
 by (\ref{functional_bound}),  it follows that
 $$ 
 E^{(\alpha)}(x,z)
 \le 
 C h^{M} \max_{|\beta|=M} \|D^{\beta}U\|_{L_{\infty}(B(z,Kh))}
 \le 
 Ch^{L-|\alpha|} \max_{|\gamma|= L}\|D^{\gamma}\phi\|_{L_{\infty}(B(x-z,Kh))}.
 $$
 We note that  $\min \{|\eta|\mid \eta\in B(x-z,Kh)\} \ge \frac12 |x-z|$,
 so by Assumption 2, we have
 $ \|D^{\gamma} \phi \|_{L_{\infty}(B(x-z,Kh))}\le C |x-z|^{s-L}$,
 which implies that
 \begin{equation*}
E^{(\alpha)}(x,z)
\le 
C h^{L-|\alpha|} |x-z|^{s-L}
= 
Ch^{s-|\alpha|}\left(\frac{|x-z|}{h}\right)^{s-L}.
\end{equation*}
 {\em Case 2:} Assume that $r_0-Kh\le|x-z|\le r_0+Kh$. 
 In this case, 
Assumption 2 guarantees continuity of $D^{\beta} \phi$ on $\R^d\setminus \{0\}$ 
 for $|\beta|\le s+d-1$.
 Thus
 $U= D^{\alpha}\phi(x-\cdot)$
has uniformly bounded derivatives of order $M=  s+d-1-|\alpha|$. In
 this case, Lemma \ref{Taylor_Lemma} guarantees
 that 
  \begin{eqnarray*} 
 E^{(\alpha)}(x,z)
& \le& 
 Ch^{s+d-1-|\alpha|}
 \max_{|\beta|=s+d-1-|\alpha|} \|D^{\beta} U\|_{L_{\infty}(B(z,Kh))}\\
&  \le& 
 Ch^{s+d-1-|\alpha|}
 \| \phi\|_{C^{s+d-1}\bigl(\overline{B(0,2r_0)\setminus B(0, r_0/2)}\bigr)}.
 \end{eqnarray*} 
{\em Case 3: } Assume that $2Kh<|x-z|<r_0-Kh$. 
Recall that item 2 of Assumption 2 states that
 $\phi(x) = u(x) + \Hom_s(x)v(x)$ in this region.
To treat this case, we consider the $u$ and $\Hom_s v$ components separately.

By Assumption 2, we have $U = D^{\alpha} u(x-\cdot)$ has smoothness $M=L-|\alpha|$ over the
set $B(x,r_0)$, which contains
$B(x,r_0-Kh)\setminus\overline{B(x,2Kh)}$.
Thus Lemma \ref{Taylor_Lemma} guarantees that
 $
 |D^{\alpha}u(x-z) - 
 \sum_{\xi\in\Xi} a(\xi,z) D^{\alpha}u(x-\xi)|
 \le 
 C h^{L-|\alpha|}
 $.
  Since $|x-z| <r_0$, we have 
 \begin{equation}
 \label{analytic_part} 
 \Bigl|
 	D^{\alpha}u(x-z) - \sum_{\xi\in\Xi} a(\xi,z) D^{\alpha}u(x-\xi)
\Bigr|
\le C r_0^{L-s} h^{s-|\alpha|} \left(\frac{|x-z|}{h}\right)^{s-L}.
\end{equation}
Similarly, letting $U = D^{\alpha} (\Hom_sv) (x-\cdot)$, Lemma \ref{Taylor_Lemma} gives
 \begin{eqnarray*}
&&
\Bigl|
	D^{\alpha} (\Hom_sv) (x-z)  - \sum_{\xi\in\Xi} a(\xi,z) D^{\alpha}(\Hom_s v)(x-\xi)
\Bigr|\\
 &&\le  
 C h^{L-|\alpha|}\max_{|\beta|=L} \|D^{\beta} (\Hom_s v)\|_{L_{\infty}(B(x-z,Kh))}.
 \end{eqnarray*} 
 because $\max_{|\beta|= L-|\alpha|} \|D^{\beta} U\|_{L_{\infty}(B(z,Kh))}\le 
 \max_{|\beta|=L} \|D^{\beta} (\Hom_s v)\|_{L_{\infty}(B(x-z,Kh))}$.
We can estimate $ \|D^{\beta} (\Hom_s v)\|_{L_{\infty}(B(x-z,Kh))}$ by using the  inequality
$$
\|D^{\beta} (\Hom_s v) \|_{L_{\infty}(B(x-z,Kh))}
\le 
C_{\beta}
 \sum_{\gamma\le \beta}  
   \| D^{\gamma} \Hom_s\|_{L_{\infty}(B(x-z,Kh))}  
   \|D^{\beta-\gamma} v\|_{L_{\infty}(B(x-z,Kh))},
$$
which follows with a $\beta$ dependent constant from the Leibniz rule. 
 By (\ref{hom_deriv}), there is a 
 constant so that for any 
 $|\gamma|\le L$,
  $|D^{\gamma} \Hom_s(x)|\le C |x|^{s-L}$ on $B(0,r_0)$.
 Since $|x-z|>2Kh$, it follows that $\inf \{|\zeta|\mid \zeta\in B(x-z,Kh)\}>\frac12|x-z|$, so
 \begin{equation}
 \label{Homogeneous_part}
 \Bigl|
 	D^{\alpha} (\Hom_sv) (x-z)  - \sum_{\xi\in\Xi} a(\xi,z) D^{\alpha}(\Hom_s v)(x-\xi)
\Bigr|
 \le  Ch^{L-|\alpha|} |x-z|^{s-L} .
 \end{equation}
 The result in Case 3 follows by combining (\ref{analytic_part}) and (\ref{Homogeneous_part}).
 \end{proof}
 \begin{lemma}
 \label{near_field}
 Suppose $\phi$ satisfies Assumption 2.
Then for $0< |\alpha|<s+d/2$, and $0<|x-z|<2Kh$, the error kernel satisfies
 $$
 E^{(\alpha)}(x,z)
 \le
 C
 \Bigl(h^{L-|\alpha|}
 +
 |x-z|^{s-|\alpha|} 
 + 
 \sum_{\xi\in\Xi} \bigl| a(\xi,z) \bigr| |x-\xi|^{s-|\alpha|} 
 \Bigr).
 $$
 \end{lemma}
 \begin{proof}
 Assumption 2 allows us to split $E^{(\alpha)}(x,z)$ into
 a totally smooth part  and a homogenous part $ E^{(\alpha)}(x,z) \le E_S +E_H$
 where
 \begin{eqnarray*}
E_S &:=&
 \Bigl|
 	D^{\alpha} u(x-z) -  \sum_{\xi\in\Xi} a(\xi,z)D^{\alpha}u(x-\xi)
\Bigr|\\
E_H &:=& 
 \Bigl|
 	D^{\alpha} (\Hom_sv)(x-z) -  \sum_{\xi\in\Xi} a(\xi,z)D^{\alpha}(\Hom_sv)(x-\xi)
\Bigr|.
 \end{eqnarray*}
 The smooth part is treated as in the proof of Lemma \ref{far_field}.
 In particular, Lemma \ref{Taylor_Lemma} ensures that 
  \begin{equation}
\label{Smooth_near_part}
\Bigl
  	|D^{\alpha} u(x-z) -  \sum_{\xi\in\Xi} a(\xi,z)D^{\alpha}u(x-\xi)
\Bigr|
\le 
Ch^{L-|\alpha|}.
\end{equation}
 To treat 
 $E_H$,
 we use the Leibniz rule and smoothness of $v$, to obtain
 \begin{eqnarray}
 \label{Hom_near_part}
E_H
& \le& 
 C  \sum _{\gamma \le \alpha}
 \bigl(|D^{\gamma} \Hom_s(x-z)|  
+
\sum_{\xi\in\Xi} \bigl| a(\xi,z) \bigr| |D^{\gamma}\Hom_s(x-\xi)|
 \bigr)\nonumber\\
&\le&
C\Bigl(  |x-z|^{s-|\alpha|} 
 + 
 \sum_{\xi\in\Xi} \bigl| a(\xi,z) \bigr| |x-\xi|^{s-|\alpha|} \Bigr),
 \end{eqnarray}
 where the second estimate follows from (\ref{hom_deriv}).
 Combining (\ref{Smooth_near_part}) and (\ref{Hom_near_part}) gives the result.
 \end{proof}
 \begin{theorem}
 \label{main_approximation}
 Suppose $\phi$ satisfies Assumption 2,
$f = \nu*\phi+p$, 
with $p\in \mathcal{P}$, and $\nu\in L_2(\R^d)$ having support in a bounded, open set $\Omega$ having
Lipschitz boundary.
Then for $\sigma$ with $\lceil\sigma\rceil<s+d/2$, the approximation error satisfies
$$
\| f-T_\Xi f \|_{{W}_2^{\sigma}(\R^d)}
\le 
C h^{s+d-\sigma}\|\nu\|_{L_2(\R^d)}.
$$
 \end{theorem}
 \begin{proof}
 We begin by considering an integer $\sigma<s+d/2$. 
 Let $\alpha$ be a multi-index with $|\alpha|=\sigma$.
Then we have
 $
 \|D^{\alpha}f- D^{\alpha}T_{\Xi} f \|_{L_2(\R^d)}
 =
 \left( \int_{\R^d} \left| \int_{\R^d} E^{(\alpha)}(x,z) \nu(z) \dif z\right|^2 \dif x\right)^{1/2}
 $
by differentiating under the integral.
Defining quantities $A$ and $B$ as  
\begin{align*}
A&:=  \Bigl\| \int_{|\cdot-z|>2Kh} E^{(\alpha)}(\cdot,z) \nu(z) \dif z\Bigr\|_{L_2(\R^d)},\\ 
B&:=  \Bigl\| \int_{|\cdot-z|<2Kh} E^{(\alpha)}(\cdot,z) \nu(z) \dif z\Bigr\|_{L_2(\R^d)},
\end{align*}
we split the error into two parts:
$ \|D^{\alpha}f- D^{\alpha}T_{\Xi} f \|_{L_2(\R^d)}\le A+B$.
This corresponds to splitting the error kernel as
$E^{(\alpha)}= E_1 +E_2$,
where
\begin{align*}
E_1(x,z)&:=E^{(\alpha)}(x,z)\chi_{\{(x,z)\mid |x-z|>2Kh\}}(x,z)\\ 
E_2(x,z)&:=E^{(\alpha)}(x,z)\chi_{\{(x,z)\mid |x-z|<2Kh\}}(x,z).
\end{align*}
We may control $E_1$ by Lemma \ref{far_field}
and $E_2$ by Lemma \ref{near_field}.

By integrating $E_1(x,z)$  with respect to either $x$ or $z$, 
we obtain an estimate  for the $L_p$ norm of the integral operator  $\mathcal{E}_1: g\mapsto \int_{\R^d} g(z)E_1(x,z) \dif z$.
 In particular, for $1\le p\le \infty$,
 \begin{eqnarray*}
 \|\mathcal{E}_1\|_{L_p(\R^d)\to L_p(\R^d)}
&\le&
 C  h^{s-|\alpha|} \int_{2Kh< |y|<r_0-Kh}\left({|y|}/{h}\right)^{s-L}\dif y\\
&&\mbox{} + 
  C  h^{s+d-1-|\alpha|} \mathrm{vol}(\{y\mid r_0-Kh< |y|<r_0+Kh\})\\
  &&\mbox{} +
  C  h^{s-|\alpha|} \int_{r_0+Kh< |y|<\infty}\left({|y|}/{h}\right)^{s-L}\dif y
\end{eqnarray*}
So  $  \|\mathcal{E}_1\|_{L_p(\R^d)\to L_p(\R^d)}\le  Ch^{s+d-|\alpha|}$.
In particular, this holds for $p=2$,
which gives
\begin{equation}
\label{I_1}
A
\le 
Ch^{s+d-|\alpha|} \|\nu\|_{L_2(\R^d)} .
\end{equation} 

By Lemma \ref{near_field},
$E_2(x,z) \le C( h^{L-|\alpha|}+|x-z|^{s-|\alpha|} + \sum_{\xi\in\Xi} \bigl| a(\xi,z) \bigr| |x-\xi|^{s-|\alpha|})$
for $x,z$ satisfying $|x-z| < 2Kh$.
This allows us to estimate $B$ with three integrals, each generated by one of the above terms.
Defining $B_1$, $B_2$ and $B_3$ as
\begin{align*}
B_1 &:= \Bigl\| \int_{|\cdot-z| < 2Kh} h^{L-|\alpha|} |\nu(z)|\dif z\Bigr\|_{L_2(\R^d)},\\ 
B_2 &:=  \Bigl\| \int_{|\cdot-z| < 2Kh}   |\cdot-z|^{s-|\alpha|}| \nu(z)|\dif z \Bigr\|_{L_2(\R^d)},\\ 
B_3 &:=  \Bigl\| 
			\int_{|\cdot-z| < 2Kh} \sum_{\xi\in\Xi} \bigl|a(\xi,z)\bigr| \, |\cdot-\xi|^{s-|\alpha|}  |\nu(z)|\dif z
		\Bigr\|_{L_2(\R^d)}.
\end{align*}
By H{\"o}lder's inequality, we then have
\begin{equation}\label{decomp}
B  \le C( B_1+B_2+B_3).
\end{equation}
%
The first two parts can be
controlled by the method used for $E_1$, giving 
\begin{eqnarray}
B_1 
&\le& 
C h^{L+d-|\alpha|}\|\nu\|_{L_2(\R^d)},\label{B_1}\\
B_2 
 &\le& C h^{s+d-|\alpha|}\|\nu\|_{L_2(\R^d)} ,\label{B_2}
\end{eqnarray} 
since $s-|\alpha|>-d/2>-d$.

To handle $B_3$,
we apply H{\"o}lder's inequality to the sum $ \sum_{\xi\in\Xi} \bigl|a(\xi,z)\bigr| \, |\cdot-\xi|^{s-|\alpha|} $, 
writing $|a(\xi,z)|=\sqrt{|a(\xi,z)|}\sqrt{|a(\xi,z)|}$, 
and $\|a(\cdot,z)\|_{\ell_1(\Xi)} = \sum_{\xi\in \Xi} |a(\xi,z)|$
to obtain
$$
(B_3)^2\le
\int_{\R^d} 
\Bigl|\int_{B(x,2Kh)} 
\|a(\cdot,z)\|_{\ell_1(\Xi)}^{1/2}
 \Bigl(
 	\sum_{\xi\in \Xi} |a(\xi,z)| \, |x-\xi|^{2(s-|\alpha|)}\, |\nu(z)|^2
\Bigr)^{1/2} 
\dif z
 \Bigr|^2\dif x.
$$
Applying H{\"o}lder's inequality to the inner integral gives
\begin{eqnarray*}
(B_3)^2
&\le&
\int_{\R^d}
\Bigl(
\int_{B(x,2Kh)} 
\|a(\cdot,\zeta)\|_{\ell_1(\Xi)}
\dif \zeta
\Bigr)\times\\
&&
\quad
\int_{B(x,2Kh)} \sum_{\xi\in \Xi} |a(\xi,z)|   \, |x-\xi|^{2(s-|\alpha|)}\, \bigl|\nu(z)\bigr|^2 \dif z\dif x.
\end{eqnarray*}
By the estimate
$
\int_{B(x,2Kh)} 
\|a(\cdot,\zeta)\|_{\ell_1(\Xi)}
\dif \zeta
\le 
\int_{B(x,2Kh)}\Gamma \dif \zeta
\le 
C h^d
$,
we have
$$
(B_3)^2\le
Ch^d
\int_{\R^d}
\int_{B(x,2Kh)} \sum_{\xi\in \Xi} |a(\xi,z)| \, |x-\xi|^{2(s-|\alpha|)}\, |\nu(z)|^2 \dif z\dif x.
$$
Because $a(\xi,z)=0$ when $|z-\xi|>Kh$ and $z\in B(x,2Kh)$,
the inner sum is taken only over $\xi\in \Xi$ which are  within  $3Kh$ from $x$.
We use this to switch the order of sums and integrals:
\begin{eqnarray*}
(B_3)^2
&\le&
Ch^d
\int_{\R^d}
\sum_{|\xi-x|<3Kh}
\, |x-\xi|^{2(s-|\alpha|)}\,
 \left(\int_{B(x,2Kh)} \bigl| a(\xi,z) \bigr| \bigl|  \nu(z) \bigr|^2 \dif z\right)\dif x\\
&\le&
 Ch^d
\sum_{\xi\in\Xi}
\left(
\int_{B(\xi,3Kh)}
\, |x-\xi|^{2(s-|\alpha|)}\,
 \dif x
 \right)
 \left(\int_{\R^d}|a(\xi,z)| \bigl|  \nu(z) \bigr|^2 \dif z\right).
\end{eqnarray*}
The last integral can be made larger by increasing the domain
of integration to $\R^d$.
At this point, we observe that 
$\int_{B(\xi,2Kh)}
 \, |x-\xi|^{2(s-|\alpha|)}\,
 \dif x\le C h^{2s-2|\alpha|+d}$.
This leaves
\begin{eqnarray}\label{B_3}
(B_3)^2
&\le& 
Ch^{2s-2|\alpha|+2d}
\int_{\R^d}
\|a(\cdot,z)\|_{\ell_1(\Xi)}
 |\nu(z)|^2 \dif z\nonumber\\
&\le&
C\Gamma h^{2s+2d-2|\alpha|}
\|\nu\|_{L_2(\R^d)}^2 
\end{eqnarray}

 The  bound $
B
\le  
C h^{s+d-|\alpha|}\| \nu\|_{L_2(\R^d)}
$
follows from the
 decomposition (\ref{decomp})  and estimates (\ref{B_1}), (\ref{B_2}) and (\ref{B_3}).
 Combining this fact with (\ref{I_1}) completes the proof in case $\sigma \in \Nat$.
 
 For fractional $\sigma$ with $\lceil\sigma\rceil< s+d$, we simply interpolate between 
 integer order Sobolev spaces, using
 $\sigma_1=\lfloor \sigma\rfloor$ and $\sigma_2=\lceil \sigma \rceil$, 
 so that $\sigma = \theta \sigma_2 +(1-\theta) \sigma_1$.
 This can be done by using H{\"o}lder's inequality to estimate the Fourier characterization of the $H^{\sigma}$ norm, or  to 
 by way of the Gagliardo-Nirenberg inequality.
 In either case, we have the estimate
 $\|F\|_{W_2^{\sigma}(\R^d)}\le C \|F\|_{W_2^{\sigma_1}(\R^d)}^{1-\theta}\|F\|_{W_2^{\sigma_2}(\R^d)}^{\theta}$,
 which ensures
 $$  \|f-T_{\Xi} f\|_{W_2^{\sigma}}
 \le
 \bigl(C h^{s+d-\sigma_1} \|\nu\|_{L_2(\R^d)} \bigr)^{1-\theta} 
 \bigl(C h^{s+d-\sigma_2} \|\nu\|_{L_2(\R^d)}\bigr)^{\theta}
 .
 $$
 The result follows because $h^{(s+d-\sigma_1)(1-\theta)} h^{(s+d-\sigma_2)\theta} =h^{s+d-\sigma}$.
 \end{proof}

 %
 %
 %
 %
\section{Interpolation with positive definite RBFs}
\label{S:main_PD}
With the aid of the Bernstein estimates from section \ref{S:Bernstein}, 
we show that the approximation rate
of Theorem \ref{main_approximation}
is inherited by RBF interpolation: 
for this, we consider 
an RBF $\phi$ having a  native space 
which is norm equivalent to $ H^{\tau}$, and 
a target function  for which the doubling result of \cite{schaback_doubling} applies.
We measure the interpolation error 
$\|f-I_{\Xi} f\|_{H^{\sigma}(\R^d)}$
for suitable values of $\sigma>0$.

\subsection{Main result for  positive definite RBFs}
\begin{theorem}\label{main_interpolation}
Suppose $\tau>d/2$ and $\phi$ is a  positive definite RBF 
with native space equivalent to the Sobolev space $H^\tau(\R^d)$.
Suppose, further, that $\phi$ satisfies Assumption 2  with $s=2\tau-d$.
If $\Omega\subset \R^d$ is compact and satisfies an interior cone condition,
then there is a constant $C$ so that the following holds.
For any $f\in H^{2\tau}(\R^d)$ 
 which satisfies 
 $f=\phi*\nu$  with
 $\nu\in L_2(\R^d)$ supported
 in $\Omega$,
for any sufficiently dense subset ${\Xi}\subset \Omega$
and
for $\sigma>0$ satisfying $\lceil \sigma\rceil < 2\tau-d/2$,
the inequality
$$
\|f-I_{\Xi}f\|_{H^{\sigma}(\R^d)}
\le 
C h^{\tau}q^{\tau-\sigma}\|\nu\|_{L_2(\R^d)}
$$
holds.
\end{theorem}
If ${\Xi}$ is such that $q$ and $h$ are kept roughly on par, e.g., 
if ${\Xi}$ is quasi-uniform with  controlled mesh ratio $\rho= h/q$,
then $\|f-I_{\Xi}f\|_{W_2^{\sigma}(\Omega)}\le C\rho^{\tau-\sigma} h^{2\tau-\sigma}\|\nu\|_{L_2(\R^d)}$. 
This extends  previous doubling results in this context, which held for $\sigma\le \tau$. 
In other words, the novelty of this theorem is that it holds in case 
$\tau<\sigma$ and $\lceil \sigma \rceil <2\tau-d/2$.
\begin{proof}
By the above comment, we consider $\sigma$ which satisfies $\sigma>\tau$ and $\lceil \sigma\rceil <2\tau-d/2$.
By hypothesis, 
$\phi$ satisfies Assumption 1. 
Thus Theorem \ref{cpd_euclidean} applies to  
$I_{\Xi}f-T_{\Xi} f\in V_{\Xi}(\phi)$,
and
$\|\opJ_{\sigma-\tau}(I_{\Xi}f-T_{\Xi} f)\|_{\N(\phi)}\le C q^{\tau-\sigma}\|I_{\Xi}f-T_{\Xi} f\|_{\N(\phi)}$
holds, which implies
$$\|I_{\Xi}f-T_{\Xi} f\|_{H^{\sigma}}\le C q^{\tau-\sigma}\|I_{\Xi}f-T_{\Xi} f\|_{\N(\phi)}.$$
Theorem \ref{main_approximation}  gives 
$\|f-T_{\Xi}f\|_{\N(\phi)} \le
C \|f-T_{\Xi}f\|_{H^{\tau}} \le C h^{\tau} \|\nu\|_{L_2(\R^d)}$,
while the standard doubling argument given in the proof of  \cite[Theorem 5.1]{schaback_doubling} 
shows that
\begin{eqnarray*}
\|f-I_{\Xi}f\|_{\N(\phi)}^2
\le 
 \|f-I_{\Xi}f\|_{L_2(\Omega)}\|\nu \|_{L_2(\R^d)}.
\end{eqnarray*}
Since the inequality  
$\|f-I_{\Xi}f\|_{L_2(\Omega)}\le C h^\tau\| f-I_{\Xi}f \|_{{W}_2^\tau(\R^d)}$
holds by standard arguments 
(see \cite[Theorem 11.32]{Wend}, or the original version \cite[Theorem 2.12]{NWW}), we have, by
norm equivalence of the spaces $\N(\phi)\sim H^{\tau}\sim W_2^{\tau}(\R^d)$, that
$$\|f-I_{\Xi} f\|_{\N(\phi)}
\le 
Ch^{\tau} \|\nu\|_{L_2(\R^d)}.$$
Thus, the triangle inequality gives
$\|I_{\Xi} f-T_{\Xi} f\|_{\N(\phi)}
\le
C h^\tau \|\nu\|_{L_2(\R^d)}$, and 
\begin{equation}\label{indirect}
\|I_{\Xi}f-T_{\Xi} f\|_{H^{\sigma}}
\le C q^{\tau-\sigma} h^{\tau} \|\nu\|_{L_2(\R^d)}
\end{equation}
follows.
On the other hand, 
 a direct application of Theorem \ref{main_approximation}   gives 
\begin{equation}\label{jackson}
\|f-T_{\Xi}f\|_{H^{\sigma}} 
\le
  C h^{2\tau-\sigma} \|\nu\|_{L_2(\R^d)}.
 \end{equation}
Together, (\ref{jackson}) and (\ref{indirect}) give
$$
\|f-I_{\Xi}f\|_{ H^{\sigma}}
\le
\|f-T_{\Xi}f\|_{H^{\sigma}}
+\|T_{\Xi} f-I_{\Xi}f\|_{H^{\sigma}} 
\le 
C h^{\tau}q^{\tau-\sigma} \|\nu\|_{L_2(\R^d)}
$$ 
and the result follows.
\end{proof}

%
%
%
\subsection{A note on compactly supported RBFs}
\label{S:compact_support}
As pointed out in Example \ref{Wendland_example_2}, 
the compactly supported RBFs of minimal degree constructed in \cite[Chapter 9]{Wend}
do not satisfy Assumption 2, unless $d=2$. 
This is precisely because of the  behavior at  the boundary of the support of $\phi_{d,k}$. 
This can be addressed by following the same construction, but using a radial polynomial of slightly higher degree. (It may
also be satisfied by other compactly supported RBFs, of which there are many,  
one may find other constructions in \cite{Buhmann,CharinaContiDyn,Wu}.) 

We recall here some aspects of Wendland's construction which
can be used to 
 construct compactly supported RBFs that satisfy Assumption 2.
 
For a measurable function, 
 $f:(0,\infty)\to \R$
which is  integrable with respect to  $\dif\mu =s\dif s$,
 we define  $\I f(r):= \int_r^{\infty}sf(s) \dif s$.  
 The operator $\I$ has an intertwining property with the Fourier
transform:
 the $d+2$-dimensional Fourier transform of a suitably integrable  radial function
 $f$ equals, as a radial function, the $d$-dimensional Fourier transform of $\I f$, see e.g \cite[Lemma 2.1]{W95}:
i.e.,  
\begin{eqnarray*}
&&r^{-d/2}\int_0^{\infty} f(t) t^{(d+2)/2} J_{d/2}(rt)\dif t \\
&& = r^{-(d-2)/2}\int_0^{\infty} \I f(t) t^{d/2} J_{(d-2)/2}(rt)\dif t  \quad  \text{for all } r>0.
\end{eqnarray*}
Define $\psi_{\ell}(0,\infty)\to \R $ by $\psi_{\ell} := (1-\cdot)_+^{\ell}$.
Then for  spatial dimension $d$, and integer $\ell\ge\lfloor d/2\rfloor +1$,
 the function $x\mapsto \psi_{\ell}(|x|) = (1-|x|)_+^{\ell}$  is radial, positive definite and supported in $B(0,1)$. 
Via Bochner's theorem and the above intertwining formula,
 $x\mapsto (\I^k \psi_{\ell})(|x|)$ is positive definite as well, see also \cite[Eq. (5)]{W95}.

The  RBFs of minimal degree described in Example \ref{Wendland_example_2} are defined
as $\phi_{k,d} := \I^k \psi_{\ell}$, with $\ell= k+\lfloor d/2\rfloor +1$.
  For general $k,\ell$,  and $f:[0,1]\to \R$,
  a simple induction gives the identity
 $\I^k f  (r) =\frac{2^{1-k}}{\Gamma(\alpha)} \int_r^1 t f(t) (t^2-r^2)^{k-1}\dif t$ 
 for $r\le1$. In particular, the family of functions
 $$\I^k \psi_{\ell}(r) := \frac{2^{1-k}}{\Gamma(\alpha)} \int_r^1 t (1-t)^{\ell} (t^2-r^2)^{k-1}\dif t$$ 
 can be extended to non-integer values of $k$ and $\ell$; such ``generalized Wendland functions'' have been
 introduced and studied in  \cite{CH}.
 
 By  collecting known results from \cite{Wend} and \cite{CH}, the following proposition shows
 that, 
 for $\ell\ge k+d$, 
 each kernel $\I^k \psi_{\ell}$ satisfies the hypotheses of Theorem \ref{main_interpolation}.
 In particular, we have the compatibility 
 between Sobolev order $m$ and homogeneity parameter $s$ from Assumption 2: 
 namely $s =  2m-d$
  since both quantities equal $2k+1$.

\begin{proposition}
 For  integers $k,\ell$ satisfying $\ell \ge k+d$,
the function
$$\psi_{\ell,k}:\R^d\to \R: x\mapsto \I^k \psi_{\ell}(|x|) $$
is a compactly supported RBF which satisfies Assumption 2 with $s=2k+1$.
Its native space, $\N(\psi_{\ell,k})$,  is norm equivalent to $ W_2^m(\R^d)$ with  $m=k+\frac{d+1}{2}$.
\end{proposition}
\begin{proof}
 Smoothing properties of the operator $\I$ given in \cite[Lemma 9.8]{Wend} guarantee
that $\I^k \psi_{\ell} \in C^{\ell+k}\bigl((0,\infty)\bigr)$. 
Since $s=2k+1$ and $\ell\ge k+d$, it follows that 
$\ell+k\ge s+d-1$, so
$\I^k \psi_{\ell} \in C^{s+d-1}\bigl((0,\infty)\bigr)$, as required. 

Because $\phi$ has support in $B(0,1)$,
item 1 holds with $r_0=1$.

Since each application of $\I$ increases the polynomial degree by 2, 
$\I^k \psi_{\ell}$ is  polynomial of degree $2k+\ell$, and, as observed in \cite[Theorem 9.12]{Wend}, 
the first $k$ odd-degree coefficients in the monomial expansion of $\I^k\psi_{\ell}$ vanish. 
This also follows directly from the formula in \cite[Theorem 3.2]{CH}.
By splitting into even and odd degree powers, we obtain
 $$\I^k\psi_{\ell}(|x|)= \sum_{j=0}^{2k+\ell} d_j |x|^{j}
=
\left(\sum_{j=0}^{k+\lfloor \ell/2\rfloor} d_{2j}|x|^{2j}\right)
+ 
|x|^{2k+1}\left( \sum_{j=0}^{\lfloor (\ell-1)/2\rfloor} d_{2k+1+2j} |x|^{2j}\right)$$
so item 2 holds with $s=2k+1$.

The fact that $\N(\psi_{\ell,k}) = W_2^{k+\frac{d+1}{2}}(\R^d)$ 
has been observed in \cite[Corollary 2.4]{CH}.
Specifically, 
the $d$-dimensional 
Fourier transform of
$\psi_{\ell,k}$ is shown to satisfy $\widehat{\psi_{\ell,k}}(\xi) \sim (1+|\xi|)^{-(d+2k+1)}$
in \cite[Eqn. (2.3)]{CH} . 
\end{proof}

%
%
%
%
%
%
%
\section{Interpolation using conditionally positive definite RBFs}
\label{S:main_CPD}
The CPD case requires an extra assumption and has a slightly different error estimate. 
For various reasons, the target function $f=\nu*\phi+p\in \N(\phi)$ must 
satisfy the 
polynomial annihilation condition 
$\nu\perp \mathcal{P}_{\m-1}$, which is equivalent to the vanishing moment condition 
$\widehat{\nu}(\xi) =\mathcal{O}(|\xi|^{\m})$. 
Furthermore, the error estimate is initially in terms of the quantity a
$\EE (f) $,
which we can be refined in a few ways (this is discussed after the proof).

Section \ref{SS:main_CPD} provides the  analogous result  to Theorem \ref{main_interpolation} for CPD kernels.
Sections \ref{SS:algebraic} and \ref{SS:surface_splines} give instances where the annihilation condition is guaranteed to hold
and provide bounds for the quantity $\EE (f) $ in terms of the fill distance.

%
%
%
\subsection{Main result for  conditionally positive definite RBFs}
\label{SS:main_CPD}
\begin{theorem}
\label{main_interpolation_CPD}
Suppose 
$\phi$ is an RBF which is 
CPD of order $\m$ 
and which satisfies
Assumptions 1 and 2, with $s=2\tau-d$. %
If $\Omega\subset \R^d$ is compact and satisfies an interior cone condition,
then there is a constant $C$ so that
for $f\in \N(\phi)$, with $f=\phi*\nu +p$ with $p\in \mathcal{P}_{\m-1}$,
 and $\nu\in L_2(\R^d)$, having
support in $\Omega$,
and $\nu\perp \mathcal{P}_{\m-1}$,
if ${\Xi}\subset \Omega$
is a
 sufficiently dense set, then
$$
|\opJ_{\sigma-\tau}(f-I_{\Xi}f)|_{\N(\phi)}
\le 
C q^{\tau-\sigma}
(h^{\tau} +
\EE(f) 
)\|\nu\|_{L_2(\R^d)}
$$
with $\lceil \sigma\rceil < 2\tau-d/2$.
Here $\EE(f):=\frac{ \|f-I_{\Xi} f\|_{L_2(\Omega)}}{|f-I_{\Xi}f|_{\N(\phi)}}$.
\end{theorem}
\begin{proof}
The estimates
\begin{equation*}
|\opJ_{\sigma-\tau}(f-T_{\Xi}f)|_{\N(\phi)} \le  C h^{2\tau-\sigma} \|\nu\|_{L_2(\R^d)}
\quad 
\text{and}
 \quad
|f-T_{\Xi}f |_{\N(\phi)} \le C h^{\tau} \|\nu\|_{L_2(\R^d)}
\end{equation*}
follow from Theorem \ref{main_approximation} and the embedding $H^{\tau}=W_2^\tau(\R^d) \subset \N(\phi)$
which implies the estimate 
$
|\opJ_{\sigma-\tau}(f-T_{\Xi}f)|_{\N(\phi)}
\lesssim 
 \|\opJ_{\sigma-\tau}(f-T_{\Xi}f)\|_{H^{\tau}}
\lesssim 
 \|f-T_{\Xi}f\|_{W_2^{\sigma}(\R^d)}$.
We can treat the interpolation error in the native space by using a similar `doubling' argument 
to that of Theorem \ref{main_interpolation}. Orthogonality gives
$$|f-I_{\Xi}f |_{\N(\phi)}^2 = \langle f,f-I_{\Xi} f\rangle_{\N(\phi)}  = \int_{\R^d} \widehat{\nu}(\omega) 
\bigl(\widehat{f}(\omega) - \widehat{I_{\Xi} f} (\omega)\bigr) \dif \omega.$$
Some care is necessary to apply a Plancherel-like result, 
since $\widehat{f}(\omega) - \widehat{I_{\Xi} f} (\omega)$ 
is only a generalized Fourier transform (and also not necessarily in $L_2(\R^d)$). 
The identity 
$$
\int_{\R^d} \widehat{\nu}(\omega) 
\bigl(\widehat{f}(\omega) - \widehat{I_{\Xi} f} (\omega)\bigr) \dif \omega 
= \int_{\Omega} {\nu}(x) 
\bigl({f}(x) - {I_{\Xi} f} (x)\bigr) \dif x$$
is handled
in Lemma \ref{Planch} below. Applying Cauchy-Schwarz
gives
$$
|f-I_{\Xi} f|_{\N(\phi)}^2 
\le
\|\nu\|_{L_2(\Omega)} 
\|f- I_{\Xi} f\|_{L_2(\Omega)}
\le
 \|\nu\|_{L_2(\Omega)} 
  |f- I_{\Xi} f |_{\N(\phi)}
  \EE (f).
 $$
 Dividing gives
$|f-I_{\Xi} f|_{\N(\phi)}\le  \EE(f) \|\nu\|_{L_2(\R^d)}$
and 
applying the triangle inequality gives
$ 
|(T_{\Xi} f -I_{\Xi} f)|_{\N(\phi)}
\le 
(C  h^{\tau}  + \EE(f) )\|\nu\|_{L_2(\R^d)}
$.
Because $\nu\perp \mathcal{P}_{\m-1}$, 
it follows from Remark \ref{annihilation_remark} that $T_{\Xi} f\in V_{\Xi}(\phi)$.
Since $T_{\Xi} f -I_{\Xi}f\in V_{\Xi}(\phi)$,
we may
apply Theorem \ref{cpd_euclidean}
to obtain
$$ 
|\opJ_{\sigma-\tau}(T_{\Xi} f -I_{\Xi}f) |_{\N(\phi)}\le  
Cq^{\tau -\sigma}  |(T_{\Xi} f -I_{\Xi}f) |_{\N(\phi)}
\le C q^{\tau-\sigma}(h^{\tau} +\EE)\|\nu\|_{L_2(\R^d)}
$$
and the result follows.
\end{proof}

Under some extra conditions on the RBF, $\EE(f)$ can be controlled by the fill distance, yielding a result
similar to the positive definite case. This is discussed below.
However, even without extra hypotheses, the term $\EE(f)$ can   be estimated by the power function
$P_{\Xi}(x) = \sup_{|f|_{\N(\phi)}=1} {|f(x)-I_{\Xi} f(x)|}$, which can be estimated by \cite[Theorem 11.9]{Wend}.
\begin{corollary}\label{PF_cor}
Suppose $\phi$ satisfies the requirements of Theorem  \ref{main_interpolation_CPD}.
Then
$$
|\opJ_{\sigma-\tau}(f-I_{\Xi}f) |_{\N(\phi)}%
\le 
C q^{\tau-\sigma}
h^{\tau-d/2}  \|\nu\|_{L_2(\R^d)}.
$$
\end{corollary}
\begin{proof}
Because $I_{\Xi}$ is idempotent,
we have
 \begin{eqnarray*}
 \EE (f)
 &=&
 \frac{ \|f-I_{\Xi}f -I_{\Xi}(f-I_{\Xi} f)\|_{L_2(\Omega)}}{ |f-I_{\Xi}f |_{\N(\phi)}}\\
 &\le&
  \frac{\mathrm{vol}(\Omega))^{1/2}  \max_{x\in\Omega}|f(x)-I_{\Xi}f(x) -I_{\Xi} (f(x)-I_{\Xi} f(x))|}
  { |f-I_{\Xi}f |_{\N(\phi)}}\\
 &\le& (\mathrm{vol}(\Omega))^{1/2} \max_{x\in\Omega} P_{\Xi}(x).
 \end{eqnarray*}
 By \cite[Theorem 11.9]{Wend},  for $\ell>\m-1$, there exist positive constants $c_1,c_2$ so that 
  $$P_{\Xi}(x) \le C \min_{p\in\mathcal{P}_{\ell}}\|\phi- p\|_{L_{\infty}(B(0,c_2h))}^{1/2}.$$
  Since $\phi$ satisfies Assumption 2 (in particular item 2) it follows that, if we take $\ell>2\tau-d$, we have
$  \min_{p\in\mathcal{P}_{\ell}}\|\phi- p\|_{L_{\infty}(B(0,c_2h))}\le Ch^{s/2} = Ch^{2\tau-d}$.
The result follows by plugging the estimate $ \EE (f)\le C h^{\tau-d/2}$ into Theorem \ref{main_interpolation_CPD}.
 \end{proof}

We now  make the additional assumption
necessary to refine $\EE(f)$ by using the zeros lemma.
We assume that $\widehat{\phi}(\xi) \le C |\xi|^{-2\tau}$
in a neighborhood of the origin, which without loss  is  
$ B(0,r_0)\setminus\{0\}$, where $r_0$ is the constant from Assumption \ref{A:CPD}
(by continuity of $\widehat{\phi})$.
Together with Assumption \ref{A:CPD}, this is equivalent  to 
assuming the continuous embedding  $\N(\phi) \subset \dot{H}^{\tau}$.
We also assume $\m\le \tau$, which permits us to compare the $W_2^{\tau}(\R^d)$ 
seminorm
appearing in  the zeros estimate
and the  homogeneous 
seminorm of $\dot{H}^{\tau}$.
%
%
%
\begin{corollary}
\label{zeros_cor}
Suppose $\phi$ satisfies the requirements of Theorem  \ref{main_interpolation_CPD}
with $\m\le \tau$
and that  $\widehat{\phi}(\xi) \le C |\xi|^{-2\tau}$ for $0<|\xi|<r_0$.
Then
$$
|\opJ_{\sigma-\tau}(f-I_{\Xi} f) |_{\N(\phi)}
\le 
C q^{\tau-\sigma}
h^{\tau} \|\nu\|_{L_2(\R^d)}.
$$
\end{corollary}
\begin{proof}
Because $\m\in \Nat$, we have $\m\le \lfloor \tau\rfloor$.
By the zeros estimate
\cite[Theorem A.4]{HNW-p}, we have 
$\|f-I_{\Xi}f\|_{L_2(\Omega)} \le C h^{\tau} |f-I_{\Xi} f|_{{W}_2^{\tau}(\R^d)}$.
 Because $f-I_{\Xi}f\in\N(\phi)$, \cite[Theorem 10.21]{Wend} ensures
 it has generalized Fourier transform of order $\m/2 \le \lfloor \tau\rfloor/2$,
 (note that  $\m$ is an integer).
 Lemma \ref{Sob_equiv} applies and guarantees that the
  $W_2^{\tau}(\R^d)$ and $\dot{H}^{\tau}$ seminorms are identical.
  Consequently, we have 
$\|f-I_{\Xi}f\|_{L_2(\Omega)} \le C h^{\tau} |f-I_{\Xi} f|_{\dot{H}^{\tau}}$.
Because $\widehat{\phi}(\xi) \le C |\xi|^{-2\tau}$ for all $\xi\neq 0$,
we have 
$$\|f-I_{\Xi}f\|_{L_2(\Omega)}\le C h^{\tau}  |f-I_{\Xi} f |_{\N(\phi)}$$
It follows that $\EE(f) \le Ch^{\tau}$, and the result follows.
\end{proof}

\begin{lemma}\label{Planch}
Suppose $\phi$, $X$ and $f=\nu*\phi+p$ satisfy the hypotheses of 
Theorem \ref{main_interpolation_CPD}.
Then 
$$\int_{\R^d} \widehat{\nu}(\omega) 
\bigl(\widehat{f}(\omega) - \widehat{I_{\Xi} f} (\omega)\bigr) \dif \omega 
= \int_{\R^d} {\nu}(x) 
\bigl({f}(x) - {I_{\Xi} f} (x)\bigr) \dif x.$$
\end{lemma}
\begin{proof}
We achieve this by mollification. 
 Let $\kappa\in C_c^{\infty}(\R^d)$
be a smooth function which equals 1 in $B(0,1/2)$ and vanishes outside of $B(0,1)$. Then 
$\widehat{\nu_R} := \kappa( \cdot/R) \widehat{\nu}$ is a smooth test function supported in $B(0,R)$ 
(since $\widehat{\nu}$ is entire), hence a Schwartz function satisfying 
$\widehat{\nu}_R(\xi) = \mathcal{O}(|\xi|^{\m})$; here we have used the polynomial annihilation assumption placed on $\nu$. 
Because $f-I_{\Xi} f\in \N(\phi)$, it has a generalized Fourier transform of order $\m/2$, so
$$\int_{\R^d} \widehat{\nu_R}(\omega) 
\bigl(\widehat{f}(\omega) - \widehat{I_{\Xi} f} (\omega)\bigr) \dif \omega 
= \int_{\R^d} {\nu_R}(x) 
\bigl({f}(x) - {I_{\Xi} f} (x)\bigr) \dif x.$$ 
Since 
$\int_{\R^d} |\widehat{\nu}(\omega)||\widehat{f}(\omega) -\widehat{I_{\Xi} f}(\omega)|\dif \omega\le
|f|_{\N(\phi} |f-I_{\Xi} f|_{\N(\phi)}<\infty$,  dominated convergence
guarantees that
$
\lim_{R\to \infty} \int_{\R^d}| \widehat{\nu}_R(\omega) -\widehat{\nu}(\omega)|
\bigl|\widehat{f}(\omega) - \widehat{I_{\Xi} f} (\omega)\bigr| \dif \omega
=0$. 
The fact that $f-I_{\Xi} f\in \N(\phi)$ also guarantees that it is continuous and has slow growth. 
Thus, for any 
compact set $K$, we have 
$$
\lim_{R\to \infty} \int_{K}| {\nu}_R(x) -{\nu}(x)|
\bigl|{f}(x) - {I_{\Xi} f} (x)\bigr| \dif x=0.
$$ 
If $K\supset \Omega$ then
${\nu}_R(x) -{\nu}(x)={\nu}_R(x)$ when $x\in\R^d\setminus K$. 
Writing $\nu_R$ as a convolution, namely
$\nu_R (x) = R^d \int \nu(y) \kappa^{\vee}(R (x-y))\dif y$,
it follows that
 $|\nu_R(x)|\le C R^d (1+R\:\dist(x,\Omega))^{-L}$, where we have used that $\kappa$ is a Schwartz function and 
 $\nu\in L_1$ is supported in 
 $\Omega$.  
 Because $f$ and $I_{\Xi} f$ have algebraic growth, 
 the estimate
 \begin{eqnarray*}
  \int_{\R^d\setminus K}| {\nu}_R(x) -{\nu}(x)|
\bigl|{f}(x) - {I_{\Xi} f} (x)\bigr| \dif x
&\le&
 C R^d \int_{\R^d\setminus K}\frac{ |x|^{m_1}}{(1+R\dist(x,\Omega))^{L}}\dif x\\
&\le&
C R^{d-L} \int_{\R^d\setminus K} |x|^{m_1-L}\dif x\to 0
\end{eqnarray*}
holds
and the lemma follows.
\end{proof}

In the next subsections, we consider two applications of Theorem \ref{main_interpolation_CPD}.
The first considers a RBF
where $\widehat{\phi}$ has an algebraic singularity at the origin which determines the order of 
conditional positive definiteness.
The second treats surface splines, considered in Examples \ref{TPS_1} and \ref{ss_example_2},
which do not satisfy the hypotheses of Theorem \ref{main_interpolation}.

%
%
%
%
%
%
%
%
\subsection{RBFs with algebraic singularities}
\label{SS:algebraic}
In this subsection we assume $\widehat{\phi}$
has a singularity similar to 
 $ |\omega|^{-\beta_0-d}$ near the origin.
If the other conditions of Theorem \ref{main_interpolation_CPD} hold,
 then Lemma \ref{moments} below shows that  $\nu\perp \mathcal{P}_{\lfloor \beta_0/2\rfloor}$.
 
Because
$|\cdot|^{-\beta_0-d+\alpha}$
is locally integrable
 if and only if
$\alpha> \beta_0$,
it follows that if
${\phi}$ has a  generalized Fourier transform of order $\m$, 
then $2\m>\beta_0$, since the function
$\omega\mapsto |\omega|^{2\m} \widehat{\phi}(\omega)$ must be locally integrable.
Consequently, if $\m$ is minimal in the sense that $\m= \lfloor \beta_0/2\rfloor+1$
 then 
$\nu\perp \mathcal{P}_{\lfloor \beta_0/2\rfloor}$ implies $\nu\perp \mathcal{P}_{\m-1}$.
 
 We note that this is sufficient to treat surface splines of order $m$ having the unconventional
 order $\m=\lfloor m-d/2\rfloor+1$; i.e., 
 with auxiliary polynomial space $\mathcal{P}_{\lfloor m-d/2\rfloor}$.
 As mentioned in Example \ref{TPS_1}, $\widehat{\phi_m} = |\xi|^{-2m}$, so $\beta_0= 2m-d$ in this case.
 The conventional situation of surface splines with CPD order $\m=m$ is treated in the next section.
 
 Another example which this section treats, which is relevant to the pseudospectral methods mentioned in the introduction,
 is the case of a differential operator like $\opL = 1-\Delta$ applied
 to $\phi_m$. In that case, one can see that Assumption 1 holds from the Fourier transform:
 $\widehat{\opL \phi_m}(\xi) = (1+|\xi|^2) |\xi|^{-2m}$, although the singularity at $0$ 
 does not match the decay at infinity. Assumption 2 holds in this case, too, as can be easily checked.
 Finally, Corollary \ref{zeros_cor} does not apply in case, because the singularity $|\xi|^{-2m}$ 
 is sharper than the decay at infinity $|\xi|^{2-2m}$. In this case, one could  use Corollary \ref{PF_cor}.
 
 %
 %
 %
\begin{lemma}\label{moments}
 Suppose $\phi$ is CPD of order $\m$
for which
there is  
a neighborhood $B(0,r)$ of the origin
where the following two conditions hold:
\begin{itemize}
\item  there is $C$ so that
$ \widehat{\phi}(\omega)  \le C |\omega|^{-\beta_0-d}$ a.e. in $B(0,r)$
\item 
$ \int_{B(0,r)} \widehat{\phi}(\xi) |\omega|^{\beta_0} |\log \omega|^{-1}\dif \omega=\infty. 
 $
 \end{itemize}
If $f\in \N(\phi)$ has the form $f= \nu*\phi+p$, with $\nu\in L_2(\R^d)$
having compact support,
then $\nu\perp \mathcal{P}_{\lfloor \beta_0/2\rfloor} $.
\end{lemma}
 Note that the above hypotheses are met if  there are constants $0<c\le C<\infty$ such that
 $c |\omega|^{-\beta_0-d} \le \widehat{\phi}(\omega) \le C |\omega|^{-\beta_0-d} $ a.e. in $B(0,r)$.
\begin{proof}
Assume without loss that $r<1.$
By \cite[Theorem 10.21]{Wend}, 
since $f\in \N(\phi)$ it  has
a generalized Fourier transform
which satisfies  $\widehat{f}/(\widehat{\phi})^{-1/2}\in L_2(\R^d)$.   
By H{\"o}lder's inequality,
\begin{eqnarray*}
&& \int_{B(0,r)} |\widehat{f}(\omega)| |\omega|^{\beta_0/2}|\log \omega|^{-1}\dif \omega\\
 && \le 
  \left(\int_{\R^d}  |\widehat{f}(\omega)|^2/\widehat{\phi}(\omega) \dif \omega\right)^{1/2} 
  \left(\int_{B(0,r)} |\omega|^{\beta_0}|\log \omega|^{-2} \widehat{\phi}(\omega) \dif \omega \right)^{1/2}.
\end{eqnarray*}
 holds, so $\omega\mapsto |\widehat{f}(\omega) |\omega|^{\beta_0/2} |\log \omega|^{-1}\in L_1(B(0,r))$. 
 
Since the support of $\nu$ is compact,
 $\widehat{\nu}$ is entire. 
Let $k\in \mathbb{N}$ be the smallest integer for which there is a multiindex $\alpha$ 
such that  $D^{\alpha} \widehat{\nu}(0)\neq 0$. By
Taylor's theorem, we can write $\widehat{\nu}(z) = H_k(z) + R(z)$, 
where $R(z) = o(|z|^k)$ as $z\to 0$.
Here $H_k$ is the Taylor polynomial of degree $k$ at 0; it happens to be homogeneous because of the minimality of $k$.

For
$\Theta_k\in C^{\infty}(\Sph^{d-1})$
defined by
$H_k (z) = |z|^k\Theta_k(z/|z|)$, 
the set
$$\mathcal{C} := \Bigl\{\zeta \in \Sph^{d-1}
\
 \mid
\
|\Theta_k(\zeta)|>
\frac12 
\|\Theta_k\|_{\infty} \Bigr\}$$
is open and nonempty.
Thus,  in the  cone  $\{z \in \R^d \mid   z/|z|\in \mathcal{C}\}$, we have that
$|H_k(z) |\ge   \frac12 
\|\Theta_k\|_{\infty}  |z|^k$. Since $R(z) = o(|z|^k)$, 
there is $r_0>0$, and a corresponding 
neighborhood
$ \mathcal{R}:=\{z\mid  |z|\le r_0, \ z/|z|\in \mathcal{C}\}$
such that that 
$$\Bigl(\forall z \in \mathcal{R}\Bigr)
\quad
|\widehat{\nu}(z)|\ge \frac{1}4   \|\Theta_k\|_{\infty} |z|^k   .$$  
Since
 $\widehat{\nu}(\omega) = \widehat{f}(\omega) / \widehat{\phi}(\omega)$,
 we have, for $\omega\in \mathcal{R}$, that
\begin{eqnarray*}
|\widehat{f}(\omega)| |\omega|^{\beta_0/2}|\log \omega|^{-1} 
&=& 
|\widehat{\nu}(\omega)|
 \widehat{\phi}(\omega) |\omega|^{\beta_0/2}|\log \omega|^{-1}\\
& \ge& 
 \frac{ \|\Theta_k\|_{\infty}}4
   \widehat{\phi}(\omega)
   |\omega|^{\beta_0/2+k} |\log \omega|^{-1}  .
 \end{eqnarray*}
 Since  $\mathcal{C}$ has positive measure, the integrability of the right hand side
 guarantees that $k> \beta_0/2$. Because $k$ is an integer,
$k\ge 1+\lfloor \beta_0/2\rfloor$, and the result follows.
\end{proof}

%
%
%
\subsection{Surface splines}
\label{SS:surface_splines}
Suppose now that $\phi_m$ is the fundamental solution to $\Delta^m$ on $\R^d$
Then $\phi_m$ is CPD order  $\m= m$, with $\N(\phi_m) = \mathrm{BL}_m(\R^d)$.
We also assume the boundary of $\Omega$ is $C^{\infty}$ (rather than merely Lipschitz),
and express its outer normal  by $\vec{n}:\partial \Omega\to \Sph^{d-1}$.
In this case, 
we replace the condition
  \begin{equation}
  \label{BL_doubling}
   f=\nu*\phi_m+p\in \mathrm{BL}_m(\R^d)
  \text{ with }
\nu\in L_2(\R^d),
\text{  }\mathrm{supp}(\nu)\subset \Omega
\text{  and }\nu\perp\mathcal{P}_{m-1}
\end{equation}
by a stronger version:
\begin{equation}
\label{smooth_BL_extension}
\text{the unique Beppo-Levi extension }
f _e
\text{ of }
f|_{\Omega}
\text{ is in }W_{2,loc}^{2m}(\R^d)
\end{equation}
which will
ensure that the conclusion of 
Theorem \ref{main_interpolation_CPD} holds.

If $f\in W_2^{2m}(\Omega)$, then
\cite[Theorem 8.2]{H_Layer} shows that 
the Beppo-Levi extension $\mathrm{BL}_m(\R^d)$ (i.e., the native space extension)
can be written as 
$f_e = \phi_m*\nu_f +p$.
Indeed, 
as described in \cite[Section 8.2]{H_Layer},  we have that 
\begin{equation}\label{ext_rep}
\nu_f*\phi_m =  \int_{\Omega} \Delta^m f(\alpha) \phi_m(\cdot - \alpha) \dif \alpha + \sum_{j=0}^{m-1} \int_{\partial \Omega} N_j f(\alpha) \lambda_{j,\alpha} \phi_m(x-\alpha)\dif \sigma(\alpha) ,
\end{equation}
where 
$$\lambda_j f = \begin{cases}\mathrm{Tr} \Delta^{j/2} f& j\text{ is even}\\ D_{\vec{n}} \Delta^{(j-1)/2} & j \text{ is odd}\end{cases}$$
and
the operators $N_j:W_2^{2m}(\Omega)\to L_2(\partial \Omega)$ are from 
\cite[Theorem 2.4]{ H_Layer}.
\begin{lemma} 
If $\Omega$ has $C^{\infty}$ boundary, then the condition  (\ref{smooth_BL_extension}) implies  (\ref{BL_doubling}).
\end{lemma}
\begin{proof}
Suppose (\ref{smooth_BL_extension}) holds. Let $\overline{\Omega}\subset B(0,R)$, for some $R>0$.
Then because $f_e\in W_2^{2m}(B(0,R))$, 
the trace theorem guarantees that 
$\lambda_k f_e\in W_2^{2m-j-1}(\partial\Omega)$ for $0\le k\le 2m-1$; 
in particular, the trace $\lambda_k^+$
 from $\Omega$
coincides with the trace $\lambda_k^-$ from $\R^d\setminus \overline{\Omega}$. 

By the jump conditions \cite[Corollary 3.4]{H_Layer}
for layer potentials 
$$V_j g= \int_{\partial \Omega} g(\alpha) \lambda_{j,\alpha}\phi_m(\cdot -\alpha) \dif \sigma(\alpha),$$
which state that $\lambda_j^+ V_jg - \lambda_j^- V_j g = (-1)^{j}  g$,
we have that $N_j f=0$.
Thus (\ref{ext_rep}) consists only of one term, 
and  $\nu_f = \Delta^m f\in L_2(\R^d)$, which is supported in $\Omega$.

Finally, \cite[Lemma 8.1]{H_Layer} guarantees that $\nu_f\perp \mathcal{P}_{m-1}$.
\end{proof}

 In  case (\ref{smooth_BL_extension}) holds,  Corollary \ref{zeros_cor} applies and
$$|\opJ_{\sigma-m}(f-I_{\Xi} f)|_{\dot{H}^m} \le Cq^{m-\sigma}h^{m}\|\nu\|_{L_2(\Omega)}.$$
 Furthermore, because $f$ and $I_X f$ have generalized Fourier transforms of order $m/2$,
 we can use Lemma \ref{Sob_equiv} to ensure that
  $|f-I_{\Xi} f|_{W_2^{\sigma}(\R^d)} \sim  |f-I_{\Xi} f|_{H^{\sigma}}$ 
  whenever $\sigma \ge m$,
  so for $m\le \sigma$ with $\lceil \sigma\rceil< 2m-d/2$, we have
  \begin{equation}
  \label{ss_error}
  |f-I_\Xi f|_{W_2^{\sigma}(\Omega)}\le C|f-I_\Xi f|_{\dot{H}^{\sigma}}\le 
  Cq^{m-\sigma}h^{m}\|\nu\|_{L_2(\Omega)}.
  \end{equation}
Here we have used that $|u|_{\dot{H}^{\sigma}} \le  |\opJ_{\sigma-m}(u)|_{\dot{H}^m} $ when $\sigma\ge m$.
In particular, if the point set $X$ is quasi-uniform with mesh ratio $\rho$, we have, with $\rho$ dependent constant
 $$|f-I_\Xi f|_{W_2^{\sigma}(\R^d)}\le Ch^{2m-\sigma}\|\nu\|_{L_2(\Omega)}.$$

\begin{remark}
A necessary and sufficient condition for $f\in W_2^{2m}(\Omega)$ to satisfy (\ref{smooth_BL_extension})
is that $f\in \cap_{j=0}^{m-1}\mathrm{ker}(N_j)$; this is  \cite[Corollary 8.3]{H_Layer}.
\end{remark}

\begin{remark}
A condition which implies (\ref{smooth_BL_extension})
 has been considered
 by Gutzmer and Melenk  in 
\cite{GM}. 
Namely, that $f$ is  satisfies {\em natural boundary conditions}: 
 \begin{equation}\label{natural}
 f\in W_2^{2m}(\Omega)\text{ and }
 D^{\alpha}f(x)=0 \text{ for }
 x\in \partial \Omega \text{ and }m\le |\alpha|\le 2m-1.
 \end{equation}
 The result \cite[Lemma 2]{GM} shows that if $f$ satisfies (\ref{natural}), then $f$ satisfies (\ref{smooth_BL_extension}).
 Thus (\ref{ss_error}) provides a higher order counterpart to their result
 then \cite[Theorem 2]{GM} shows that for sufficiently dense $\Xi\subset \Omega$,
 $$|f-I_{\Xi} f|_{{W}_2^k(\Omega)}\le h^{2m-k} |f|_{{W}_2^{2m}(\Omega)}$$
 holds for $k\le m$.  We note that the results of \cite{GM} hold under more general conditions, namely for $\Omega$ having
 Lipschitz boundary without  the assumption of quasi-uniformity on $\Xi$.
 \end{remark}

%
%
%
%
%
\appendix
\section{Regular  local polynomial reproductions}
\label{appendix}
 \begin{lemma} If $\Omega\subset \R^d$ is compact and satisfies an interior cone condition, 
then for every $L>0$, there exists a constant $K$ depending on $L$ and the cone aperture,  
and $h_0>0$ depending
on $L$ and both cone parameters, so that for
any finite subset
$\Xi\subset \Omega$ with $h(\Xi,\Omega)<h_0$  there is a stable, local polynomial reproduction of order $L$.
I.e.,  there
is a map $a(\cdot,\cdot): \Xi \times \Omega \to \R$ which satisfies the following four conditions:
\begin{enumerate}
\item for every $z\in \Omega$ if $\dist(\xi,z)>Kh$ then $a(\xi,z)=0$
\item for every $z\in \Omega$, $\sum_{\xi\in \Xi} \bigl| a(\xi,z) \bigr|\le 3$
\item for every $p\in \mathcal{P}_{L}$ and $z\in \Omega$, $\sum_{\xi\in \Xi }a(\xi,z) p(\xi) =p(z)$
\item for every $\xi\in \Xi $, $a(\xi,\cdot)$ is   smooth.
\end{enumerate}
\end{lemma}
\begin{proof}
Let $N:=\dim\mathcal{P}_{L}$. Select a basis $\{p_j\mid 1\le j\le N\}$ for $\mathcal{P}_{L}$.
Then the result
 \cite[Lemma 3.14]{Wend} 
guarantees the existence of a map $\tilde{a}$ which satisfies items 1-3.
Indeed,
for every $z\in \Omega$, $\sum_{\xi\in \Xi} \bigl| \tilde{a}(\xi,z) \bigr|\le 2$ and if  
$\dist(\xi,z)>\tilde{K}h$ then $\tilde{a}(\xi,z)=0$.

Let $K=\tilde{K}+1$.
Pick $y\in\Omega$. 
Let $\Xi_0:= \Xi \cap B(y, Kh)$. 
Because $\Xi_0$ is unisolvent, it contains a unisolvent subset $\Xi^{\flat}\subset \Xi_0$ 
with $\#\Xi^{\flat}=N$ (i.e., it contains a subset which is poised for interpolation by $\mathcal{P}_{L}$).
Enumerate $\Xi^{\flat}:=\{\xi_1,\dots ,\xi_N\}$, and let
  $\Xi^{\sharp}:=\Xi_0\setminus \Xi^{\flat}$.

Consider now the $C^{\infty}$ function $F:\R^d\times \R^N\to \R^N$ defined by 
$$
\bigl(F(x,b)\bigr)_j = p_j(x) - \sum_{k=1}^N b_k p_j(\xi_k) - 
\sum_{\zeta\in \Xi^{\sharp}} \tilde{a}(\zeta,y)p_j(\zeta).
$$
For $b_0:=\tilde{a}(\cdot,y)\left|_{\Xi^{\flat}}\right.$,
  $F(y, b_0)=0$, and 
$D_b F(x,b) =\bigl(\frac{\partial F_j}{\partial b_k}\bigr) = \bigl(p_j(\xi_k)\bigr)$ 
 is the Vandermonde matrix for $\Xi^{\flat}$ and 
 is therefore nonsingular for all $x$. 
By the implicit function theorem, there is $B(y,r_1)$ and
a smooth function
$g:B(y,r_1)\to \R^N$ so that
$g(y) =b_0= \tilde{a}(\cdot,y)\left|_{\Xi^{\flat}}\right. $ and
 $F(x,g(x)) =0$ for all $x\in B(y,r_1)$.

Note that 
$\|g(y)\|_{\ell_1(\R^N)}
\le 
2-\|\tilde{a}(\cdot,y)\left|_{\Xi^{\sharp}}\right.\|_{\ell_1(\Xi^{\sharp})}$.  
 It follows from continuity of $g$ that there is $r_2\in(0,r_1)$ so that we  for all $x\in B(y,r_2)$  
 $$\|g(x)\|_{\ell_1(\R^N)} \le 3-\|\tilde{a}(\cdot,y)\left|_{\Xi^{\sharp}}\right.\|_{\ell_1(\Xi^{\sharp})}.$$ 
  By decreasing the radius even further, i.e., taking $r(y):=\min(r_2,h)$,
  we have that 
  for every $x\in B\bigl(y,r(y)\bigr)$ 
  and for every $\xi\in \Xi_0$, we have
  $\dist(x,\xi) \le Kh$, since $\dist(y,\xi)< \tilde{K}h$.

For $x\in B(y,r)$, set 
$$a_y(\xi,x):=
 \begin{cases}
 \bigl( g(x)\bigr)_j & \xi = \xi_j\in \Xi^{\flat}\\ 
\tilde{a}(\xi,y)& \xi \in \Xi^{\sharp}\\
0&\xi \in \Xi\setminus \Xi_0.\end{cases}$$
and note that $a_y$ is a local polynomial reproduction of order $L$, locality $\tilde{K}$ and  stability 3 in $B(y,r)$.

 By compactness, there is a finite cover of the form $\Omega = \bigcup_{j=1}^M B(y_j,r(y_j))$. Denote by
  $a_j:\Xi \times \Omega\to \R$ the  extension by zero of  
  $a_{y_j}:\Xi \times B(y_j,r(y_j))\to \R$. 
  Let $\bigl(\psi_j\bigr)_{j=1\dots M}$ be a smooth partition of unity subordinate to this cover: i.e.,
consisting of functions $\psi_j:\Omega\to[0,1]$ with $\mathrm{supp}(\psi_j)\subset B(y_j,r(y_j))$ and
$\sum_{j=1}^M \psi_j =1$.

 Then $a:\Xi\times \Omega\to \R$ defined by
 $a(\xi,z):=\sum_{j=1}^M \psi_j(z) a_j(\xi,z)$
 is a smooth local polynomial reproduction, since 
 $$
 \sum_{\xi\in \Xi} p(\xi) a(\xi,z) = 
 \sum_{j=1}^M \psi_j(z) \sum_{\xi\in \Xi} p(\xi)a_j(\xi,z)
 =
 \sum_{
	 \substack{j=1\\ z\in B(y_j,r(y_j))}
 }^M 
 	\psi_j(z) p(z)
 = p(z).
 $$
 We have also that $\sum_{\xi\in \Xi} |\sum_{j=1}^M \psi_j(z) a_j(\xi,z)|\le 3$, 
 so $a$ 
 has stability constant $\Gamma\le 3$.
 Finally, for $z\in \Omega$, if $a(z,\xi)\neq 0$, then  for some $j$,
  $z\in B(y_j,r(y_j))$ and $a_j(\xi,z)\neq0$.
  But this implies that $|z-\xi|\le Kh$.
\end{proof}

\section*{Acknowledgment}
The authors wish to thank the anonymous reviewers whose comments helped improve this manuscript.
\bibliographystyle{amsplain}

\end{document}